\documentclass[reqno, 12pt]{amsart}
\usepackage{amsmath,amssymb,mathrsfs,amsthm,amsfonts,mathtools}
\usepackage[inline]{enumitem} 				
\usepackage[usenames,dvipsnames]{xcolor}
\usepackage{graphicx} 
\usepackage[letterpaper, margin=1in]{geometry}
\definecolor{hypercolor}{HTML}{003399}
\usepackage{hyperref}
 	
\hypersetup{colorlinks=true, linkcolor=hypercolor,citecolor=hypercolor}
\usepackage[margin=3pt, font=small]{caption}	
\newtheorem{thm}{Theorem}[section]
\newtheorem{lem}[thm]{Lemma}
\newtheorem{prop}[thm]{Proposition}
\newtheorem{cor}[thm]{Corollary}
\theoremstyle{definition}

\newtheorem{conj}[thm]{Conjecture}

\numberwithin{equation}{section}
\allowdisplaybreaks

\makeatletter
\newcommand{\mylabel}[2]{#2\def\@currentlabel{#2}\label{#1}}
\makeatother
\newcommand{\betaloc}{\beta\text{-}\mathrm{loc}}
\newcommand{\nw}{\mathrm{nw}}				
\newcommand{\comple}{\mathrm{c}}			
\newcommand{\sprime}{{\scriptscriptstyle\prime}}

\newcommand{\R}{\mathbb{R}}
\newcommand{\Z}{\mathbb{Z}}

\newcommand{\Csp}{\mathscr{C}}
\newcommand{\Lsp}{\mathcal{L}}
\newcommand{\Psp}{\mathscr{P}}				

\newcommand{\mom}{\mathbb{L}}				
\newcommand{\momshe}{L_{\scriptscriptstyle\mathrm{SHE}}}
\newcommand{\ratekpz}{I_{\scriptscriptstyle\mathrm{KPZ}}}

\newcommand{\hf}[1]{\mathsf{f}_{\star#1}}	
\newcommand{\hfc}[1]{\mathsf{g}_{\star#1}}	

\newcommand{\PP}{\mathbf{P}}				
\newcommand{\EE}{\mathbf{E}}				
\newcommand{\Ebm}{\mathbb{E}_{\scriptscriptstyle\mathrm{BM}}}	
	
\newcommand{\Ebb}{\mathbb{E}_{\scriptscriptstyle\mathrm{BB}}}

\newcommand{\eventhh}{\mathcal{E}}			
\newcommand{\eventhhi}{\mathcal{D}}			
\newcommand{\eventinc}{\mathcal{A}}			

\newcommand{\termt}{\mathfrak{t}}			
\newcommand{\termts}{\termt\hspace{1pt}}	
\newcommand{\intermt}{\mathfrak{t}^{\sprime}}
\newcommand{\xx}{\mathbf{x}}				
\newcommand{\vecxx}{\vec{\xx}}				
\newcommand{\xxi}{\mathbf{x}^{\sprime}}		
\newcommand{\vecxxi}{\vec{\mathbf{x}}{}^{\hspace{1pt}\sprime}}
\newcommand{\mm}{\mathfrak{m}}				
\newcommand{\vecmm}{\vec{\mm}}				
\newcommand{\mmi}{\mm^{\hspace{1pt}\sprime}}
\newcommand{\vecmmi}{\vecmm{}^{\sprime}}	

\newcommand{\hv}{\mathbf{r}}				
\newcommand{\vechv}{\vec{\hv}}				
\newcommand{\hvi}{\hv^{\hspace{1pt}\sprime}}
\newcommand{\vechvi}{\vec{\hv}{}^{\hspace{1pt}\sprime}}
\newcommand{\hvspace}{\mathscr{R}}			
\newcommand{\hvspacec}{\hvspace_\mathrm{conc}}

\newcommand{\sgn}{\mathrm{sgn}}				
\newcommand{\quant}{\mathfrak{X}}			

\newcommand{\ii}{i}							
\newcommand{\jj}{j}						
\newcommand{\cc}{\mathfrak{c}}				

\newcommand{\optimal}{\boldsymbol{\xi}}		

\newcommand{\bb}{\mathfrak{b}}

\renewcommand{\aa}{\mathfrak{a}}			
\newcommand{\Aranch}{\mathfrak{C}}			
\newcommand{\tangent}{\mathbf{y}}			

\newcommand{\ind}{\mathbf{1}}				
\renewcommand{\d}{\mathrm{d}}
\newcommand{\noise}{\eta}					
\newcommand{\parab}{\mathsf{p}}				

\newcommand{\f}{\mathsf{f}}
\newcommand{\g}{\mathsf{g}}
\newcommand{\h}{\mathsf{h}}
\newcommand{\mps}{\h_\mathrm{\star}}		
\newcommand{\mpss}[1]{\h_\mathrm{\star #1}}	

\newcommand{\hh}{h}							
\newcommand{\ZZ}{Z}							
\newcommand{\WW}{W}							
\newcommand{\hk}{\mathsf{q}}				
\newcommand{\X}{X}							
\newcommand{\bm}{B}							

\newcommand{\Dist}{\mathrm{Dist}}			
\newcommand{\norm}[1]{\Vert #1\Vert}
\newcommand{\Norm}[1]{\big\Vert #1\big\Vert}

\newcommand{\ip}[1]{\langle #1\rangle}

\newcommand*{\Cdot}{{\raisebox{-0.5ex}{\scalebox{1.8}{$\cdot$}}}} 

\usepackage{newtxtext}

\graphicspath{{./figs/}}

\title{Spacetime limit shapes of the KPZ equation in the upper tails}
\author{Yier Lin}
\author{Li-Cheng Tsai}

\address[Yier Lin]{\hspace{1.5pt} Department of Statistics, University of Chicago}
\address[Li-Cheng Tsai]{\hspace{1.5pt} Department of Mathematics, University of Utah}

\subjclass[2020]{}%

\begin{document}
\begin{abstract}
We consider the $n$-point, fixed-time large deviations of the KPZ equation with the narrow wedge initial condition.
The scope consists of concave-configured, upper-tail deviations and a wide range of scaling regimes that allows time to be short, unit-order, and long.
We prove the $n$-point large deviation principle and characterize, with proof, the corresponding spacetime limit shape.
Our proof is based on the results --- from the companion paper \cite{tsai2023high} --- on moments of the stochastic heat equation and utilizes ideas coming from a tree decomposition.
Behind our proof lies the phenomenon where the major contribution of the noise concentrates around certain corridors in spacetime, and we explicitly describe the corridors.
\end{abstract}

\maketitle

\section{Introduction}
\label{s.intro}
This paper studies the Kardar--Parisi--Zhang (KPZ) equation \cite{kardar86}
\begin{align}
	\label{e.kpz}
	&&
	\partial_t \hh = \tfrac12 \partial_{xx} \hh + \tfrac12 (\partial_{x} \hh)^2 + \noise,
	&&
	\hh=\hh(t,x),
	\ \ 
	(t,x)\in (0,\infty)\times\R,
\end{align} 
where $\noise=\noise(t,x)$ denotes the spacetime white noise.
The equation describes the evolution of a randomly growing interface $\hh(t,x)$, is a central model in nonequilibrium statistical mechanics, and has been widely studied in mathematics and physics; we refer to \cite{quastel2011introduction,corwin2012kardar,quastel2015one,chandra2017stochastic,corwin2019} for reviews on the mathematical literature related to the KPZ equation.

We prove the $n$-point, fixed-time Large Deviation Principle (LDP) for the KPZ equation with the narrow wedge initial condition and characterize, with proof, the spacetime limit shape.
Let $\hh_\nw$ be the solution of \eqref{e.kpz} with the narrow wedge initial condition, defined in Section~\ref{s.intro.scaling}.
Let $T$ be the scale of time, let $N$ be another scaling parameter, and scale $\hh_\nw$ as
\begin{align}
	\label{e.hscaled}
	\hh_{N}(t,x) := \tfrac{1}{N^2T} ( \hh_{\nw}(Tt,NTx) + \log\sqrt{T}).
\end{align}
This scaling is natural, as will be explained in Section~\ref{s.intro.scaling}.
The only assumptions on $N$ and $T$ are
\begin{align}
	\label{e.scaling}
	N\to\infty,
	\qquad
	N^2 T = N^2 T_N \to \infty.	
\end{align}
We take $T=T_N$ for the convenience of notation; our results actually hold for $T=A/N^2$ with $(N,A)\to(\infty,\infty)$.
Note that \eqref{e.scaling} allows $T_N\to 0$, $T_N\to 1$, and $T_N\to\infty$; see Section~\ref{s.intro.scaling} for a discussion on the scaling regimes.
Fix $\xx_1<\ldots<\xx_n$ and consider the event
\begin{align}
	\label{e.eventhh}
	\eventhh_{N,\delta}(\vechv)
	=
	\eventhh_{N,\delta}(1,\vecxx,\vechv)
	:=
	\big\{ |\hh_{N}(1,\xx_\cc)-\hv_\cc|\leq \delta, \cc=1,\ldots,n \big\}.
\end{align}
Postpone the definitions of the space of $n$-point deviations $\hvspacec$, the rate function $\ratekpz$, and the limit shape $\mps$ to Section~\ref{s.intro.ratefn}.
Roughly speaking, $\hvspacec$ consists of concave-configured, upper-tail deviations.
Let $\norm{f}_{\Lsp^\infty(\Omega)}:=\sup_{(t,x)\in\Omega}|f(t,x)|$.
We state the main result now.
\begin{thm}\label{t.main}
Notation as above and under \eqref{e.scaling}.
For any $\vechv\in(\hvspacec)^\circ$ and $R<\infty$,
\begin{align}
	\label{e.t.ldp}
	&\limsup_{\delta\to 0} \limsup_{N\to\infty}
	\Big| \frac{1}{N^3T_N} \log \PP[\eventhh_{N,\delta}(\vechv)] - \ratekpz(\vechv) \Big| 
	= 0,
\\
	\label{e.t.mps.}
	&\limsup_{\delta\to 0} \limsup_{N\to\infty}
	\frac{1}{N^3T_N} \log \PP\Big[ 
		\norm{\hh_N-\mps}_{\Lsp^\infty([\frac{1}{R},1]\times[-R,R])} > \tfrac{1}{R}  \ \Big| \ \eventhh_{N,\delta}(\vechv) 
	\Big]
	< 0.
\end{align}
\end{thm}

\subsection{Scaling, SHE, initial condition}
\label{s.intro.scaling}
Let us explain the scaling \eqref{e.hscaled}.
Common wisdom in the study of the KPZ equation says that the quadratic term $(\partial_x\hh)^2$ in \eqref{e.kpz} dominates the large-scale behaviors of the equation.
We take the scaling \eqref{e.hscaled} to satisfy the relation $(\text{height scale})\cdot T=(\text{space scale})^2$, which ensures that the quadratic term remains invariant after scaling:
$
	\partial_t \hh_N = \tfrac{1}{2N^2T} \partial_{xx} \hh_N + \tfrac12 (\partial_{x} \hh_N)^2 + \tfrac{1}{\sqrt{N^5T^2}} \noise.
$
The relation fixes the relative scaling between height and space, and $N$ parameterizes an additional degree of freedom.

Let us examine the conditions in \eqref{e.scaling} in short, unit, and long time.
In short time, the regime given by $N\to \infty$ and $N^2T=1$ correspond to the Freidlin--Wentzell or weak-noise LDP for the KPZ equation; see Section~\ref{s.intro.literature} for a literature review.
The deviations probed under \eqref{e.scaling} are those (much) larger than the Freidlin--Wentzell ones.
In unit-order time, $T\to 1$, \eqref{e.scaling} covers the deviations at \emph{all} scales much larger than one. 
In long time, $T\to\infty$, the commonly consider \textbf{hyperbolic scaling regime} corresponds to $N=1$, and the deviations probed under \eqref{e.scaling} are those (much) larger than the hyperbolic ones.

Hereafter, we write $T_N=T$ to simplify notation and assume $N\to\infty$ is taken under \eqref{e.scaling}.

We work with the Hopf--Cole solution of~\eqref{e.kpz}.
Consider the Stochastic Heat Equation (SHE)
\begin{align}
	\label{e.she}
	\partial_t \ZZ = \tfrac12 \partial_{xx} \ZZ + \noise \ZZ,
\end{align}
and define the Hopf--Cole solution as $\hh:=\log\ZZ$.
The narrow-wedge initial condition for the KPZ equation is defined as $\hh_{\nw}:=\log\ZZ_{\mathrm{delta}}$, where $\ZZ_{\mathrm{delta}}$ solves the SHE \eqref{e.she} with $\ZZ_{\mathrm{delta}}(0,\Cdot)=\delta_0$, and $\ZZ_{\mathrm{delta}}|_{(0,\infty)\times\R}>0$ thanks to \cite{mueller91,morenoflores2014strict}.
In short time, $\ZZ_\mathrm{delta}$ is controlled by the heat kernel $\hk(t,x):=(1/\sqrt{2\pi t})\exp(-x^2/(2t))$.
Under the scaling of consideration, it becomes $(N^2T)^{-1}\log\hk(Tt,NTx) = -x^2/(2t)-(N^2T)^{-1}\log\sqrt{2\pi t}-(N^2T)^{-1}\log\sqrt{T}$.
The second last term tends to zero, while the last term is countered in \eqref{e.hscaled}.

\subsection{Rate function, space of $n$-point deviations, limit shape}
\label{s.intro.ratefn}
We begin by recalling the hydrodynamic limit of $\hh_N$.
Let $\parab(t)=\parab(t,x):=-x^2/(2t)$.
In the hyperbolic scaling regime ($N=1$ and $T\to\infty$), $\hh_N(t,x)$ converges (say, in probability) to $\parab(t,x)-t/24$ \cite{amir10}.
Under \eqref{e.scaling}, applying the scaling in \eqref{e.hscaled} gives $(\parab(Tt,NTx)-Tt/24)/(N^2T) = \parab(t,x) + o_N(1)$, so the hydrodynamic limit of $\hh_N$ should just be $\parab$.

Let us define the rate function and space of $n$-point deviations.
To streamline notation, we consider a general $t\in(0,1]$, though Theorem~\ref{t.main} takes $t=1$ only.
For $\hv_\cc\geq \parab(t,\xx_\cc)$, $\cc=1,\ldots,n$, let $\hf{,t,\vecxx,\vechv}=\hf{}$ be the piecewise $\Csp^1$ function on $\R$ characterized by the properties: $\hf{}(\xx_\cc)=\hv_\cc$, for all $\cc$; $\hf{}\geq \parab(t)$; $\hf{}(x)=\parab(t,x)$ for all $|x|$ large enough; $\partial_x\hf{}$ is constant on $\{\hf{}>\parab(t)\}\setminus\{\xx_1,\ldots,\xx_n\}$; $\hf{}$ is $\Csp^1$ except at $\xx_1,\ldots,\xx_n$.
See Figure~\ref{f.shape.termt}.
The rate function is
\begin{align}
	\label{e.ratekpz}
	\ratekpz(t,\vecxx,\vechv)
	:=
	\int_{\R} \d x \, \big( \tfrac12( \partial_x \hf{,t,\vecxx,\vechv})^2 - \tfrac12(\partial_x\parab(t))^2 \big).
\end{align}
This rate function coincides with that of a Brownian Motion (BM) that evolves in $x$ conditioned to stay above the parabola $\parab(t,x)=-x^2/(2t)$.
Such a coincidence can be understood from the perspective of Gibbs line ensembles \cite{corwin2014brownian,corwin2016kpz}, though we arrive at \eqref{e.ratekpz} through taking the Legendre transform of the moment Lyapunov exponents.
Evaluating $\ratekpz$ at one point recovers the $3/2$ power law: $\ratekpz(t,0,\hv)=(4/3)(2/t)^{1/2}\,\hv^{3/2}$, for $\hv\geq0$.
This rate function should describe all upper-tail deviations in
\begin{align}
	\label{e.hvspace}
	\hvspace(t,\vecxx) := \{ \vechv=(\hv_\cc)_{\cc=1}^n : \hv_\cc \geq \parab(t,\xx_\cc), \forall \cc \}.
\end{align}
Our method relies on \emph{positive} moments and can access the subspace
\begin{align}
	\label{e.hvspacec}
	\hvspacec(t,\vecxx) := \{ \vechv\in\hvspace(t,\vecxx) : \hf{,t,\vecxx,\vechv} \text{ is concave} \},
\end{align}
whose interior is $\hvspacec(t,\vecxx)^\circ=\{ \vechv: \hv_\cc>\parab(t,\xx_\cc),\forall\cc; \ \hf{,t,\vecxx,\vechv} \text{ is strictly concave} \}$.
When $t=1$ and $\vecxx$ has been fixed, we write $\ratekpz(\vechv)$ and $(\hvspacec)^\circ$ to simplify notation, like in Theorem~\ref{t.main}.

To define the limit shape, consider the integral of Burgers equation and its backward version
\begin{align}
	\label{e.iburgers}
	&&&&\partial_t \h &= \tfrac12 (\partial_x \h)^2,
	&&
	(t,x)\in(0,1)\times\R.
\\
	\label{e.iburgers.back}
	&&&&\partial_s\h(1-s,x) &= -\tfrac12 (\partial_x \h(1-s,x))^2,
	&&
	(s,x)\in(0,1)\times\R.
\end{align}
Being fully nonlinear, \eqref{e.iburgers} can have multiple weak solutions under a given initial condition, but has a unique entropy solution, and similarly for \eqref{e.iburgers.back}.
Fix $\vechv\in(\hvspacec)^\circ=\hvspacec(1,\vecxx)^\circ$ and let $\hf{}:=\hf{,1,\vecxx,\vechv}$ be as before.
The limit shape $\mps$ is the entropy solution of the \emph{backward} equation \eqref{e.iburgers.back} with the terminal condition $\mps(1,\Cdot)=\hf{}$.
It is readily checked that a weak solution of \eqref{e.iburgers.back} is automatically a weak solution of \eqref{e.iburgers}.
On the other hand, an entropy solution of \eqref{e.iburgers.back} is in general \emph{not} an entropy solution of \eqref{e.iburgers}.
In particular, $\mps$ is a non-entropy solution of \eqref{e.iburgers}.
See the second last paragraph in Section~\ref{s.intro.mechanism} for a heuristic explanation why the limit shape should be a non-entropy solution of \eqref{e.iburgers}.

\begin{figure}
\begin{minipage}[b]{.48\linewidth}
\fbox{\includegraphics[width=\linewidth]{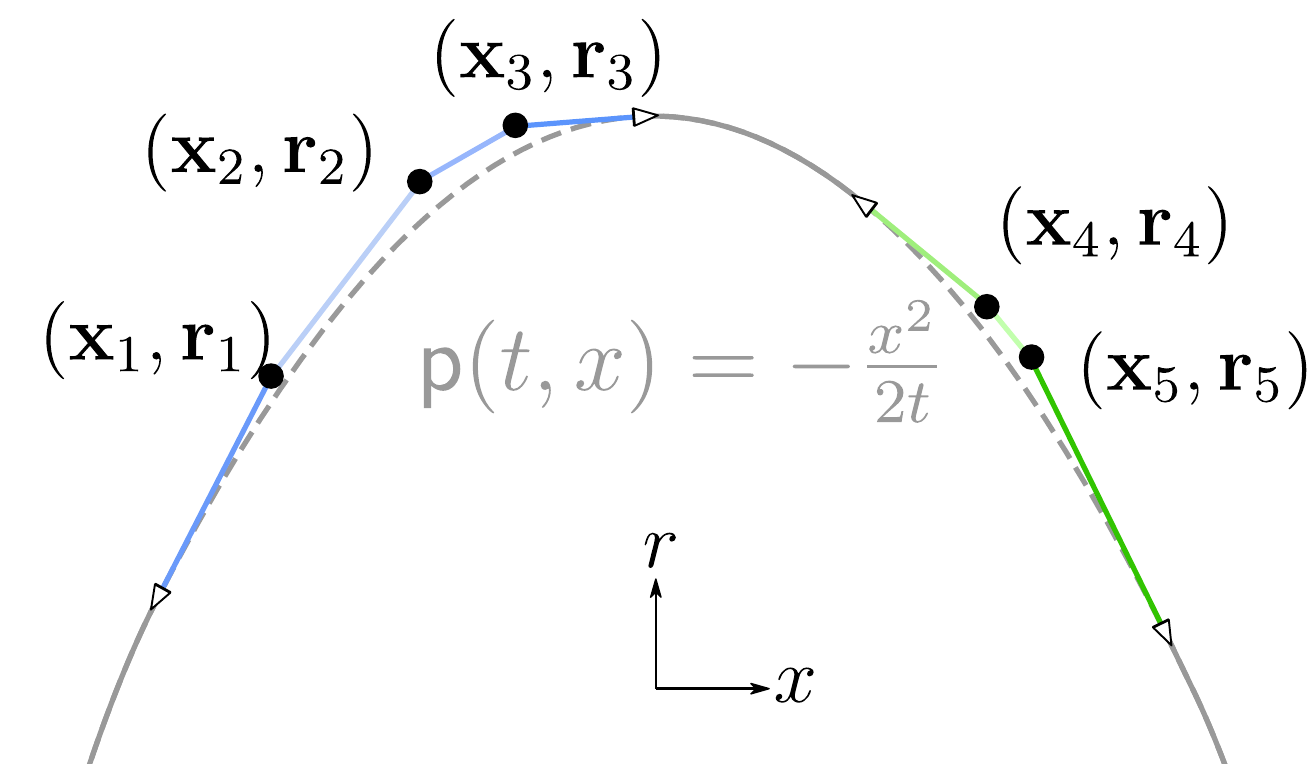}}
\caption{$\hf{}=\hf{,t,\vecxx,\vechv}$}
\label{f.shape.termt}
\end{minipage}
\hfill
\begin{minipage}[b]{.5\linewidth}
\fbox{\includegraphics[width=\linewidth]{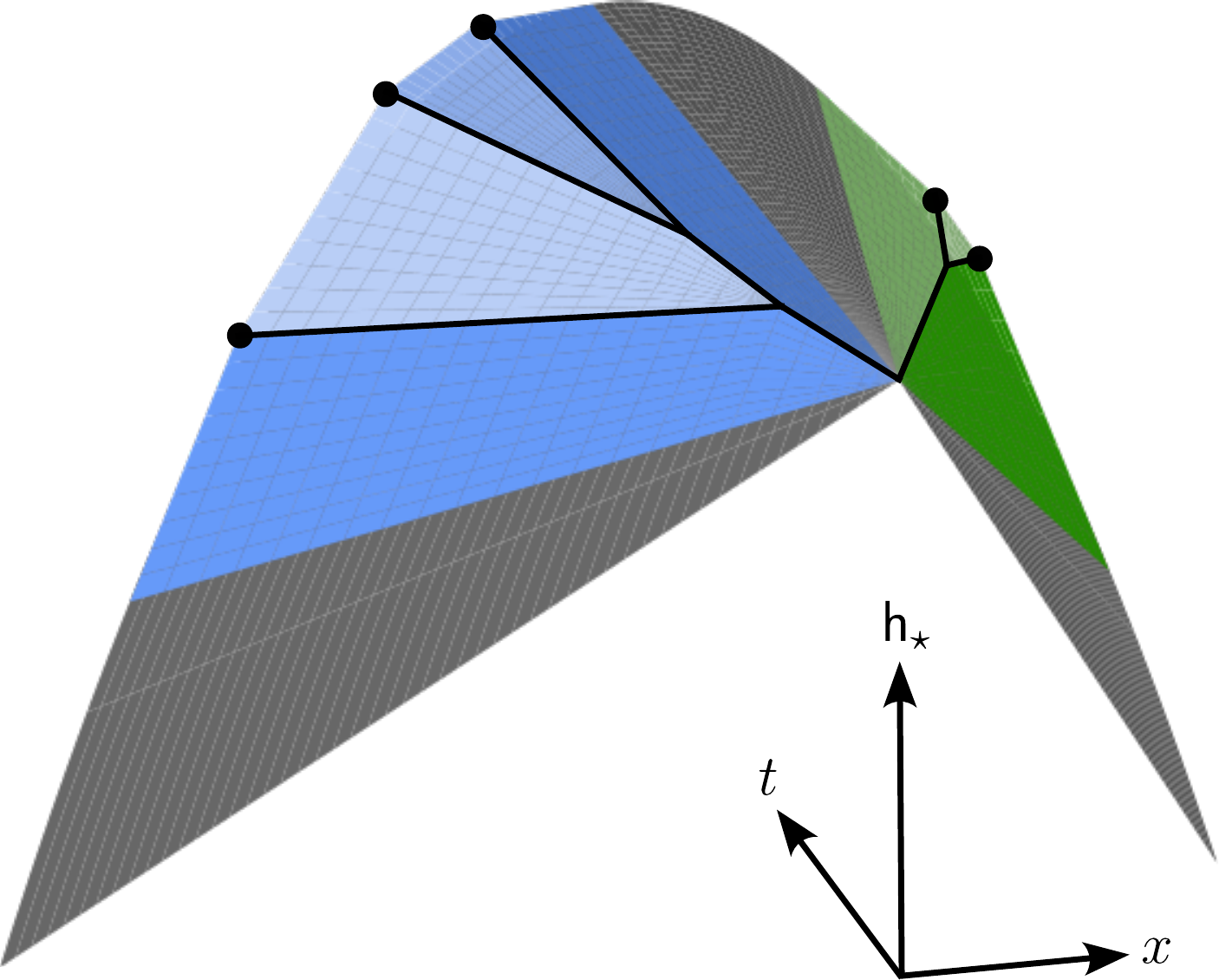}}
\caption{Limit shape $\mps$}
\label{f.shape}
\end{minipage}
\end{figure}
\begin{figure}
\begin{minipage}[b]{.52\linewidth}
\fbox{\includegraphics[width=\linewidth]{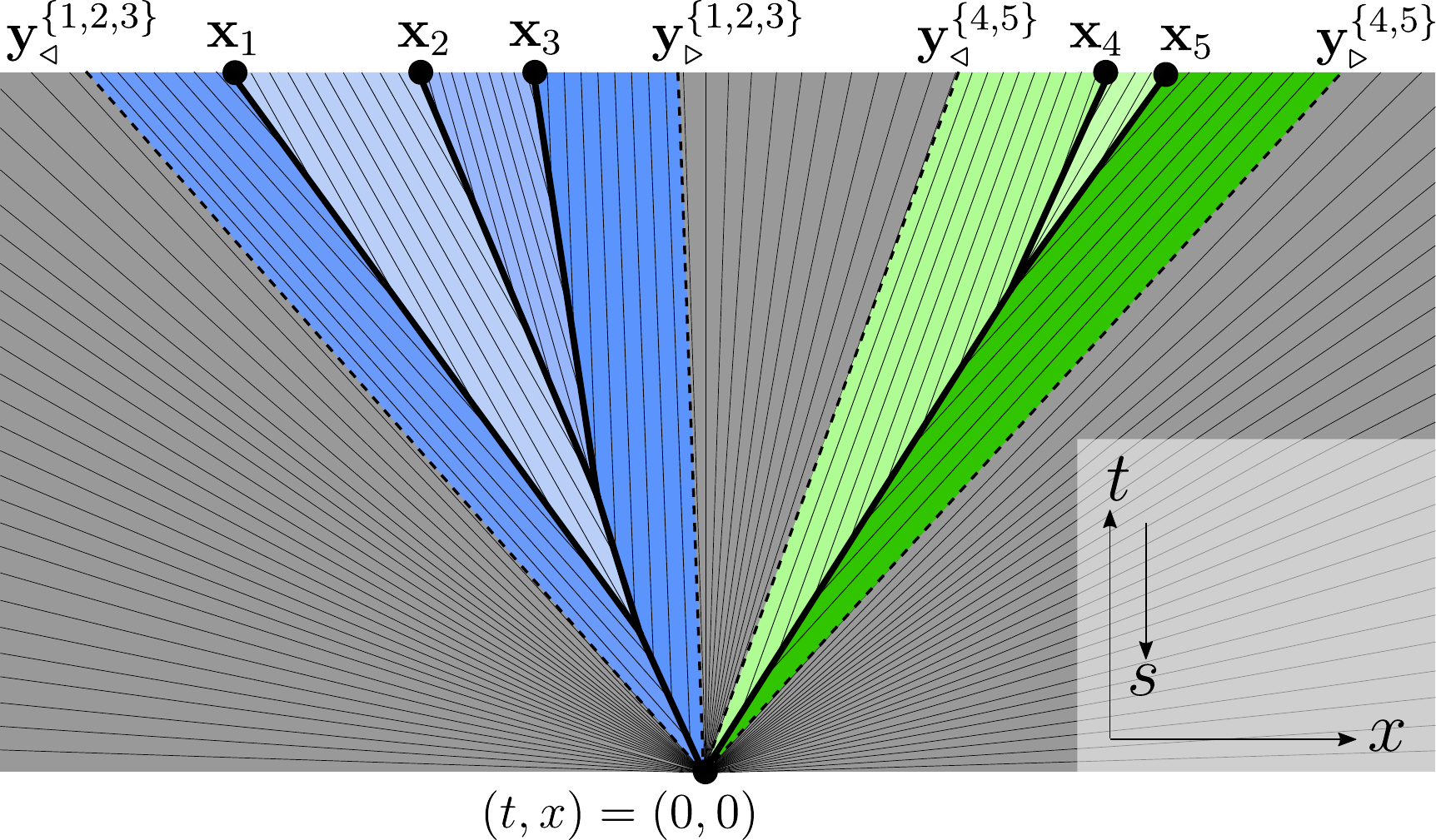}}
\caption{Shocks, aka noise corridors, (thick solid lines); characteristics (thin solid lines)\\ \hspace{1pt}}
\label{f.shocks}
\end{minipage}
\hfill
\begin{minipage}[b]{.46\linewidth}
\fbox{\includegraphics[width=\linewidth]{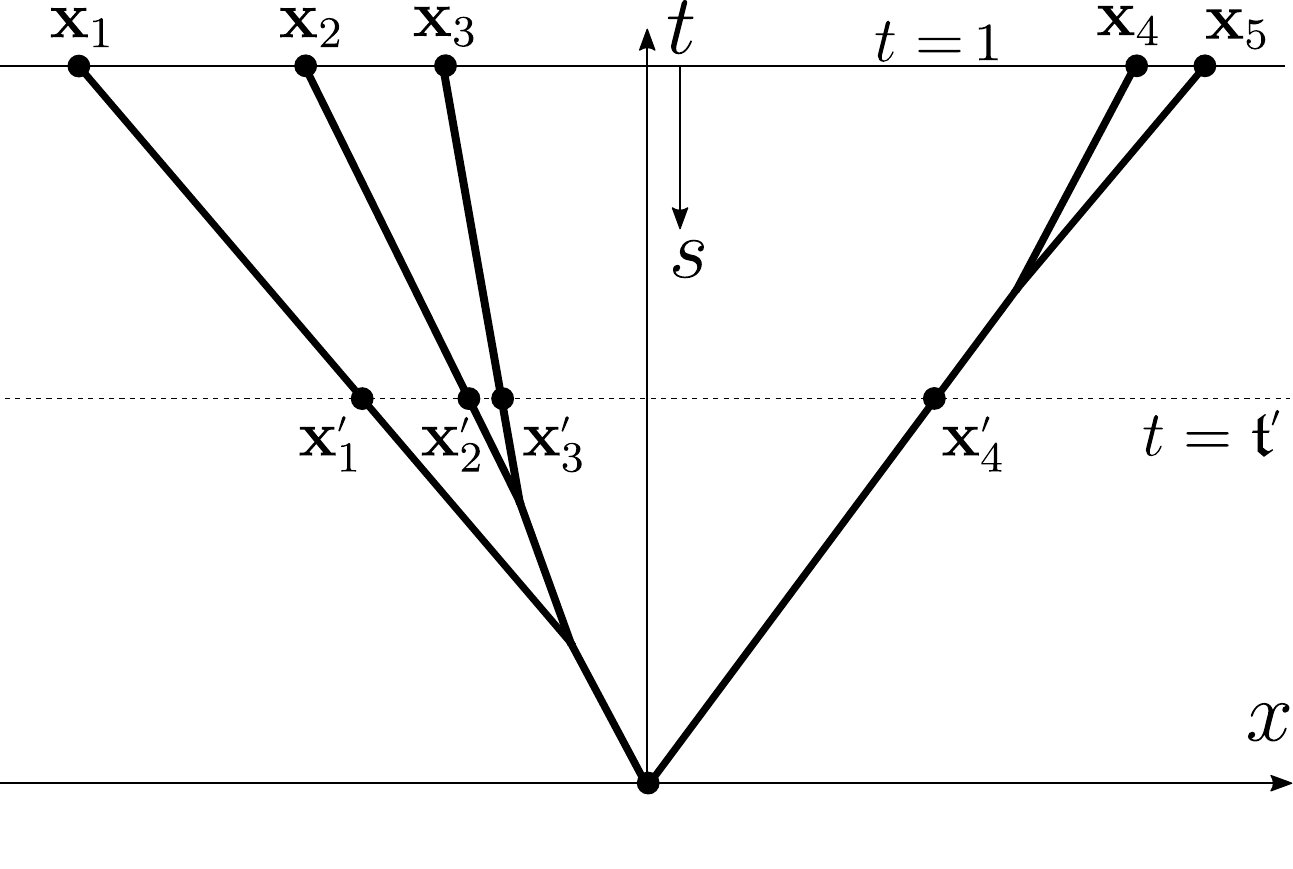}}
\caption{Intermediate time configuration.
In this figure, $\Aranch(1)=\{1\}$, $\Aranch(2)=\{2\}$, $\Aranch(3)=\{3\}$, and $\Aranch(4)=\{4,5\}$.}
\label{f.intermt}
\end{minipage}
\end{figure}

Let us describe a few properties of $\mps$, which are not hard to verify; see \cite[Section~8.3]{tsai2023high} for example.
The (trajectories of the) shocks are piecewise linear; when viewed in \emph{backward} time, they start at the $\xx_\cc$s, can merge, but never branch.
Let $\optimal_\cc=\optimal_\cc(s)$ denote the shocks, in \emph{backward} time $s$.
Partition $\{1,\ldots,n\}$ into intervals $\bb$ such that $\cc,\cc'\in\bb$ if and only if $\optimal_\cc$ and $\optimal_{\cc'}$ merge within $s\in[0,1)$.
There exist $\tangent^\bb_{\triangleleft}<\tangent^\bb_{\triangleright}$ such that each shock $\optimal_\cc$ with $\cc\in\bb$ is contained in the cone $\{(t,x): \tangent^\bb_{\triangleleft} t \leq x\leq\tangent^\bb_{\triangleright} t\}$; the cones themselves are disjoint.
Outside the cones, $\mps=\parab$.
Within each cone, $\mps$ is piecewise linear, and the kinks occur exactly along the shocks. 
See Figure~\ref{f.shape}.
Recall that the characteristics are lines, in spacetime, along which $\partial_x\mps$ is constant.
Outside the cones, the characteristics are lines emanating from $(t,x)=(0,0)$. 
Within each region where $\mps$ is linear, the characteristics are parallel lines.
Further, when viewed in forward (resp.\ backward) time, the characteristics always emerge from (resp.\ merge into) the shocks.
See Figure~\ref{f.shocks}.

\subsection{Mechanism of the LDP, noise corridors}
\label{s.intro.mechanism}
Let us explain the physical mechanism behind the LDP.
The explanation is not fully rigorous and intended to give an overall understanding only.

Consider first the case $n=1$ and $\xx_1=0$.
Use the Feynman--Kac formula to express $\hh_N$ as
\begin{align}
	\label{e.feynmankac}
	\hh_N(t,x) = \tfrac{1}{N^2T}\log \Ebm\big[ e^{\int_0^{Tt} \d s\, \noise(Tt-s,\X(s))} \delta_0(\X(Tt)) \big],
\end{align}
where $\X=$ (standard BM)$+NTx$, and the exponential is in the Wick sense.
The formula~\eqref{e.feynmankac} suggests that the ``larger'' the noise $\noise$ is, the larger $\hh_N$ becomes.
Note that this description is not rigorous because $\noise$ is not function-valued.
The target value $\hv_1=\hv$ for $\hh_N(1,0)$ exceeds the typical value $\parab(1,0)$.
Hence the noise would make itself anomalously large so that $\hh_N(1,0)$ reaches $\hv$, and, in doing so, the noise would seek the ``most efficient strategy''.
The most efficient strategy should be having $\noise$ anomalously large in a small region around $[0,1]\times\{0\}=\{(t,x):x=0\}$ in spacetime.
This strategy is efficient because $\noise$ needs only to perform anomaly in a small region; this strategy is effective because, in \eqref{e.feynmankac}, those realizations of the BM that stay around $x=0$ will pick up this anomaly and make $\hh_N(1,0)$ large.
We call $[0,1]\times\{0\}$ the \textbf{noise corridor} and call the phenomenon that the noise localizes around the noise corridor the \textbf{noise-corridor effect}.

At the level of \emph{one-point} LDPs, the noise-corridor effect has been observed and noted.
The use of subadditive structures in proving upper-tail LDPs \cite{seppalainen1998large,deuschel99,georgiou2013,janjigian15,janjigian19} (implicitly) demonstrated the noise-corridor effect.
The effect was pointed out under the name of the weak-coupling LDP regime in \cite{ledoussal16long}.
The analysis of the upper-tail asymptotics of the most probable shapes in the Freidlin--Wentzell LDP for the KPZ equation \cite{kolokolov07,kolokolov09,kamenev16,meerson16,krajenbrink21,krajenbrink22flat,tsai22,gaudreaulamarre2023kpz} also demonstrated the noise-corridor effect.

Here, we take one step further and consider $n$-point deviations.
A similar noise-corridor effect would occur: The noise would become anomalously large only in small regions around $\{(t,\optimal_\cc(1-t))\}_{t\in[0,1],\cc=1,\ldots,n}$.
We call (the trajectories of) $\optimal_\cc$, $\cc=1,\ldots,n$, the \textbf{noise corridors}.
The noise corridors are defined as the shocks in this paper, which is convenient but perhaps not intuitive.
In \cite{tsai2023high}, an alternative definition (via certain attractive Brownian particles) was offered and shown to be equivalent to the shock definition.
To the best of our knowledge, our work is the first to describe the noise-corridor effect beyond one-point LDPs.

The noise-corridor effect offers a heuristic explanation why $\mps$ should be a non-entropy solution of \eqref{e.iburgers}.
The effect says that the noise $\noise$ should be dormant in most of spacetime.
Informally removing $\noise$ from \eqref{e.kpz} and performing the scaling $\hh_N(t,x)=(N^2T)^{-1}(\hh_\nw(Tt,NTx)+\log\sqrt{T})$ give \eqref{e.iburgers}.
The solution $\mps$ is non-entropic because of the behaviors of $\noise$ around the noise corridors, and the behaviors produce artificial/non-entropy shocks along the corridors.
Put it differently, even though the noise disappears in the limiting equation \eqref{e.iburgers}, it leaves a lasting effect by producing non-entropy shocks.
This explanation comes from the picture given by Jensen and Varadhan \cite{jensen00,varadhan04}.

We emphasize that our approach is \emph{very different} from that of Jensen and Varadhan.
In the Jensen--Varadhan picture, the space of process-level deviations consists of weak solutions of \eqref{e.iburgers}, and the process-level rate function measures how ``non-entropic'' a solution is.
The major challenge in implementing the Jensen--Varadhan picture is to understand the space of all weak solutions of \eqref{e.iburgers}.
Due to this challenge, the proof of the full LDP remains largely open and has only been achieved for the TASEP by the use of exact formulas \cite{quastel21}.
Our approach, on the other hand, is to postulate what the limit shape is and to verify the postulate, by using a tree structure (described in Section~\ref{s.intro.moment}) and information extracted from the moments of the SHE.
In doing so, we \emph{bypass} the need to consider general --- potentially pathological --- weak solutions.
Our approach does not rely on integrability or explicit formulas.
Part of the reason is technical: The explicit formulas of the moments of the SHE do not seem to provide information for proving localization results such as Property~\ref{property.loc} in Section~\ref{s.momt.tsai2023high}.
More importantly, we hope that our approach can shed light on the physical mechanism of forming the deviations and limit shape.

\subsection{Ingredients of the proof, tree decomposition}
\label{s.intro.moment}
The starting point of our proof is the fixed-time, $n$-point moment Lyapunov exponent of the SHE.
Recall that $\exp(\hh) = \ZZ$, so the $n$-point moment Lyapunov exponent of the SHE is exactly the $n$-point moment generating function of $\hh$.
The results from \cite{tsai2023high} show that the exponent exists for all positive powers and characterize it.
Thanks to the condition $N\to\infty$ in \eqref{e.scaling}, characterizing the exponent for all positive powers reduces to the same task for positive \emph{integer} powers only (see \eqref{e.momentL} and the discussion after it).
Via the Legendre transform, the results on the $n$-point moment Lyapunov exponent immediately give the $n$-point LDP for the KPZ equation.

To obtain the \emph{spacetime} limit shape requires going beyond fixed-time LDPs.
The main novelty of our proof is a way of leveraging the fixed-time moment Lyapunov exponents to the spacetime limit shape.
We accomplish this through the idea of a tree decomposition, which is described next.

To describe the decomposition, we associate the moment $\EE[\prod_{\cc=1}^n \ZZ_\mathrm{delta}(T,NT\xx_\cc)^{N\mm_\cc}]$ with the tree given by the trajectories of the noise corridors.
The root of the tree is at $(t,x)=(0,0)$, which corresponds to the delta initial condition.
The leaves are at $(t,x)=(1,\xx_1),\ldots,(1,\xx_n)$, which correspond to the points where we probe the deviations.
The leaf $(1,\xx_\cc)$ carries mass $\mm_\cc$, which is the power of the moment at $(1,\xx_\cc)$, in the post-scale units.
Any other node on the tree carries the sum of the masses of its children.
Let $0 \xrightarrow{\scriptscriptstyle 1}(\vecxx,\vecmm)$ denote this weighted tree.
Take any intermediate time $\intermt\in(0,1)$, let $\vecxxi=(\xxi_\aa)_\aa$ denote the nodes of the tree at time $\intermt$, let $\mmi_\aa$ denote the mass of $(\intermt,\xxi_\aa)$, and let $\Aranch(\aa)$ be the set of indices of those leaves that are the offspring of $(\intermt,\xxi_\aa)$; see Figure~\ref{f.intermt}.
The tree decomposition says that
\begin{align*}
	\big(\text{moment of }0 \xrightarrow{1}(\vecxx,\vecmm)\big)
	&\text{ should be approximately equal to }
\\
	&\big(\text{moment of }0 \xrightarrow{\intermt}(\vecxxi,\vecmmi)\big)
	\cdot
	\prod_{\aa}
	\big(\text{moment of }\xxi_\aa \xrightarrow{1-\intermt}(\vecxx_\cc,\vecmm_\cc)_{\cc\in\Aranch(\aa)}\big),
\end{align*}
which is natural in light of the noise-corridor effect.
This decomposition is useful because it allows us to access the information at any intermediate time $\intermt$ through the right hand side, which involves fixed-time moments only.

The idea of the tree decomposition enters the first and most important step of our proof of the limit shape, which is carried out in Section~\ref{s.shape.corridors}.
In this step, we prove that $\hh_N$ concentrates around $\mps$ at any $(t,x)$ along the noise corridors, namely $(t,x)=(\intermt,\xxi_\aa)$, for some $\aa$.
Roughly speaking, we use the idea to argue that, if one removes the contribution (within an $n$-point moment of the SHE) of $h_N(\intermt,\xxi) $ around the value $\mps(\intermt,\xxi) $, the result becomes exponentially smaller. 
We will not attempt to formulate the tree decomposition itself precisely, but rather \emph{use the idea} to carry out the proof only.

Once the first step is accomplished, in Section~\ref{s.shape.outcorridors}, we infer the value of $\hh_N$ outside the noise corridor from its value along the noise corridors and its increment along the characteristics. 

\subsection{Hyperbolic scaling regime:\ discussions and potential extensions}
\label{s.intro.hyperbolic}
In long time, $T\to\infty$, perhaps the most natural scaling regime is $N=1$, which we call the hyperbolic scaling regime.

In the hyperbolic regime, the one-point upper-tail LDP and tail bounds have been obtained for various initial and boundary conditions in \cite{corwin20lower,corwin20general,das21,lin21half,ghosal20}; the bounds in the first two works actually holds for all time larger than any given positive threshold.
The recent work \cite{ganguly22} went beyond one-point deviations and studied the terminal-time limit shape.
This work gave many detailed and sharp bounds that hold for all time larger than any given positive threshold.
When specialized into the hyperbolic regime and $\vechv\in\hvspacec$, their results would produce the fixed-time, $n$-point LDP and give one-sided bounds for the terminal-time limit shape. (When $n=1$, their results give the full terminal-time limit shape; see Theorem~9 there.) 
For comparison, our result gives the full \emph{spacetime} limit shape $\mps=\mps(t,x)$ but does not cover the hyperbolic regime, because of the condition $N\to\infty$ in \eqref{e.scaling}.

Let us point out the potential of combining our methods and the results in \cite{ganguly22} to study the limit shapes in the hyperbolic regime.
Recall the Legendre duality between the rate functions and log moment generating functions.
In the hyperbolic regime, one can seek to infer from the results in \cite{ganguly22} the $n$-point LDP and to turn the LDP into the $n$-point moment Lyapunov exponent for \emph{all positive powers}.
As said, our proof can be viewed as a way of leveraging the fixed-time moment Lyapunov exponents to the spacetime limit shape.
Hence, in the hyperbolic regime, one can use these positive-power moment Lyapunov exponents as the input (that replaces Property~\ref{property.limit} in Section~\ref{s.momt.tsai2023high}) to run our proof.
A careful examination of our proof reveals that the proof goes through except where we use Property~\ref{property.loc}.
The property is used to argue for the localization of the contribution of the noise in the moments and is currently only known to hold under \eqref{e.scaling}.
If one can extend or replace this property, Theorem~\ref{t.main} for $N=1$ will follow.

\subsection{Beyond Theorem~\ref{t.main}, discussions and conjectures}
\label{s.intro.beyond}
First, we conjecture that the results of Theorem~\ref{t.main} hold for all $\vechv\in\hvspace$ under $N\geq1$ and $N^2T\to\infty$.
The limit shape should still be given by the entropy solution of the \emph{backward} equation \eqref{e.iburgers.back} with the terminal condition $\hf{}=\hf{,1,\vecxx,\vechv}$, which was defined for a general $\vechv\in\hvspace=\hvspace(1,\vecxx)$ in Section~\ref{s.intro.ratefn}.
The (conjectural) limit shape for $\vechv\notin\hvspacec$ looks more complicated than that for $\vechv\notin\hvspace$.
For example, the shocks of the former are not piecewise linear.
\begin{conj}
The results of Theorem~\ref{t.main} hold for all $\vechv\in\hvspace$ under $N\geq1$ and $N^2T\to\infty$.
\end{conj}

Next, let us consider initial conditions other than the narrow-edge/delta one.
Since the SHE \eqref{e.she} is linear, the solution with a general initial condition can be written as the convolution of the delta-initial-condition solution with the initial condition.
This property has been used to derive one-point rate functions and tail probability bounds for general initial conditions in \cite{corwin20general,ganguly22,ghosal20}.
The situation becomes more intriguing for questions about limit shapes.
As seen in the study of the Freidlin--Wentzell LDP for the KPZ equation, the most probable shapes came exhibit symmetry breaking under the Brownian initial condition \cite{janas16,smith18,krajenbrink17short,krajenbrink22flat}.

To illustrate the symmetry breaking in a cleaner fashion, we consider the two-delta initial condition.
First, if the initial condition is $\ZZ(0,\Cdot)=\delta_{NT\xx_0}$, the LDP and limit shape follow the same description in Section~\ref{s.intro.ratefn} with suitable adaptations.
Namely, we should replace $\hf{,1,\vecxx,\vechv}$ with 
$\hf{,\xx_0;1,\vecxx,\vechv}(t,x):=\hf{,1,\vecxx-(\xx_0,\ldots,\xx_0),\vechv}(x-\xx_0)$, replace $\parab$  with
$\parab_{\xx_0}(t,x):=\parab(t,x-\xx_0)$, and define the rate function $\ratekpz(\xx_0;t,\vecxx,\vechv)$ and limit shape $\mpss{,\xx_0}$ similarly.
Now, consider the two-delta initial condition $\ZZ(0,\Cdot)=\delta_{-NT}+\delta_{+NT}$ and probe the upper tails of $\hh_N(1,x)$ at $x=0$.
More precisely, consider the event $\eventhh_{N,\delta}(1,0,\hv)$ for $\hv>\parab_{-1}(1,0)=\parab_{+1}(1,0)$.
Under $\eventhh_{N,\delta}(1,0,\hv)$, we expect to see a competition between the two deltas.
Let $\mpss{,\pm 1}$ be the limit shape when there is only one delta at $\xx_0=\pm 1$ (and for $\xx=0$ and the given $\hv$).
Set $\g_-:=\max\{\mpss{,- 1}, \parab_{+ 1}\}$ and $\g_+:=\max\{\mpss{,+ 1}, \parab_{- 1}\}$.
Note that $\g_-\neq\g_+$, $\g_-(t,x)=\g_+(t,-x)$, and $\ratekpz(-1;1,0,\hv)=\ratekpz(+1;1,0,\hv)$.

\begin{conj}\label{conj.symmetrybreaking}
Under the two-delta initial condition and \eqref{e.scaling} (or more generally $N\geq 1$ and $N^2T\to\infty$), for any $\hv > \parab_{-1}(1,0)=\parab_{+1}(1,0)$ and $R<\infty$,
\begin{align}
	&\limsup_{\delta\to 0} \limsup_{N\to\infty}
	\Big| \frac{1}{N^3T}\log\PP\big[ \eventhh_{N,\delta}(1,0,\hv) \big] + \ratekpz(\pm 1;1,0,\hv) \Big|
	=
	0,
\\
	\label{e.symmetry.breaking}
	&\limsup_{\delta\to 0} \limsup_{N\to\infty}
	\Big| \PP\big[\, \Norm{ \hh_N - \g_- }_{\Lsp^\infty([\frac{1}{R},1]\times[-R,R])} \leq \tfrac{1}{R} \, \big| \, \eventhh_{N,\delta}(1,0,\hv) \big] - \frac12\Big|
	=
	0,
\\
	\tag{\ref*{e.symmetry.breaking}'}
	&\limsup_{\delta\to 0} \limsup_{N\to\infty}
	\Big| \PP\big[\, \Norm{ \hh_N - \g_+ }_{\Lsp^\infty([\frac{1}{R},1]\times[-R,R])} \leq \tfrac{1}{R} \, \big| \, \eventhh_{N,\delta}(1,0,\hv) \big] - \frac12\Big|
	=
	0.
\end{align}
\end{conj}

In Appendix~\ref{s.a.symmetrybreaking}, we analyze the one-point moments under the two-delta(-like) initial condition.
The analysis shows that the moments are dominated by the contribution of either of the deltas, which support the symmetry breaking stated in \eqref{e.symmetry.breaking}.

\subsection{Literature}
\label{s.intro.literature}
Recently, there has been much interest in the LDPs of the KPZ equation in mathematics and physics.
Several strands of methods produce detailed information on the one-point tail probabilities and the one-point rate function.
This includes the physics works \cite{ledoussal16short,ledoussal16long,krajenbrink17short,sasorov17,corwin18,krajenbrink18half,krajenbrink18simple,krajenbrink18systematic,krajenbrink18simple,krajenbrink19thesis,krajenbrink19linear,ledoussal19kp}, the simulation works \cite{hartmann18,hartmann19,hartmann20,hartmann21}, and the mathematics works \cite{corwin20general,corwin20lower,cafasso21airy,das21,das21iter,kim21,lin21half,cafasso22,ghosal20,ganguly22,tsai22exact}.
For the Freidlin--Wentzell/weak-noise LDP, behaviors of the one-point rate function and the corresponding most probable shapes for various initial conditions and boundary conditions have been predicted \cite{kolokolov07, kolokolov08, kolokolov09, meerson16, kamenev16, meerson17, meerson18, smith18exact, smith18finitesize, asida19, smith19}, some of which have been proven \cite{lin21,lin22,gaudreaulamarre2023kpz}; an intriguing symmetry breaking and second-order phase transition has been discovered in~\cite{janas16,smith18} via numerical means and analytically derived in \cite{krajenbrink17short,krajenbrink22flat}; and a connection to integrable PDEs is established and explored in \cite{krajenbrink20det,krajenbrink21,krajenbrink22flat,tsai22,krajenbrink23}.
The one-point upper-tail of the KPZ equation and the SHE have been studied in \cite{chen15,corwin20general,das21,das21iter,lin21half} for large time or all time larger than any given positive threshold, and in \cite{gaudreaulamarre2023kpz} for short time.
For the nonlinear generalizations of the SHE, finite-time upper tails have been obtained in \cite{conus2013chaotic,chen2015moments,khoshnevisan17}.
The work \cite{ganguly22} gave detailed and sharp bounds on the $n$-point upper tails and estimates of the terminal-time limit shape, for all time larger than any given positive threshold. 

\subsection{Notation}
\label{s.intro.notation}
Let us introduce the various versions of $\ZZ$ and $\hh$ that we will work with.
Below, we will use the Feynman--Kac formula to introduce the $\ZZ$s, with $\X=\text{(standard BM)}+NTx$.
For any initial time $t_0\in[0,1]$ and initial location $x'\in\R$, 
\begin{align}
	\label{e.ZZ.ptp}
	\ZZ_{N}(t_0,x';t,x)
	&:=
	\Ebm\big[ e^{\int_{0}^{T(t-t_0)}\d s\, \noise(Tt-s,\X(s))} \delta_{NTx'}(\X(T(t-t_0))) \big],
\\
	\label{e.ZZ.soften}
	\ZZ_{N,\alpha}(t_0,x';t,x)
	&:=
	\Ebm\big[ e^{\int_{0}^{T(t-t_0)}\d s\, \noise(Tt-s,\X(s))} \ind_{[-\alpha+x',x'+\alpha]}(\tfrac{1}{NT}\X(T(t-t_0))) \big].
\end{align}
Namely, $\ZZ(Tt_0,NTx';t,x):=\ZZ_N(t_0,x';t/T,x/(NT))$ solves the SHE on $(t,x)\in[Tt_0,\infty)\times\R$ with the initial condition $\ZZ(Tt_0,NTx';Tt_0,\Cdot)=\delta_{NTx'}$.
The same holds for $\ZZ_{N,\alpha}(t_0,x';t,x)$, except that the initial condition is now changed to 
$
	\ZZ_{N,\alpha}(t_0,x';t_0,\Cdot) = \ind_{[-\alpha+x',x'+\alpha]},
$
which we call the \textbf{delta-like} initial condition.
Recall the noise corridors $\optimal_\cc$ from Sections~\ref{s.intro.ratefn}--\ref{s.intro.mechanism}.
Let $\Csp(\Omega)$ denote the space of real-valued continuous functions on $\Omega$.
Set
\begin{align}
	\Dist_{N,[0,s']}(f,\optimal)
	&:=
	\sup_{s\in[0,s']} \ \min_{\cc=1,\ldots,n} \big| \tfrac{1}{NT}f(Ts) - \optimal_\cc(s) \big|,
	\quad
	f\in \Csp[0,Ts'],
\\
	\label{e.ZZ.ptp.loc}
	\ZZ^{\betaloc}_{N}(t_0,x';t,x)
	&:=
	\Ebm\big[  \text{(same as \eqref{e.ZZ.ptp})}\,\ind_{\{ \Dist_{N,[0,t-t_0]}(\X,\optimal(\cdot+1-t)) \leq \beta\}} \big],
\\
	\label{e.ZZ.soften.loc}
	\ZZ^{\betaloc}_{N,\alpha}(t_0,x';t,x)
	&:=
	\Ebm\big[ \text{(same as \eqref{e.ZZ.soften})}\,\ind_{\{ \Dist_{N,[0,t-t_0]}(\X,\optimal(\cdot+1-t)) \leq \beta\}} \big].
\end{align}
These $\beta$-localized $\ZZ$s have their BMs stay within distance $\beta$ (in the post-scale units) from the noise corridors.
When $(t_0,x')=(0,0)$, we write
\begin{align}
	\label{e.ZZ.omit.initial}
	&\ZZ_{N}(t,x) := \ZZ_{N}(0,0;t,x),
	\text{ and the same for } \ZZ_{N,\alpha}, \ZZ^{\betaloc}_{N} , \ZZ^{\betaloc}_{N,\alpha},
\\
	\label{e.hhN}
	&\hh_{N}(t,x) := \tfrac{1}{N^2T}(\log\ZZ_{N}(t,x)+\log\sqrt{T}),
	\qquad
	\hh_{N,\alpha}(t,x) := \tfrac{1}{N^2T}\log\ZZ_{N,\alpha}(t,x),
\end{align}
which is consistent with \eqref{e.hscaled}.
Recall from Section~\ref{s.intro.scaling} that $\log\sqrt{T}$ counters the contribution from the delta initial condition, so is not needed for $\hh_{N,\alpha}$.

Throughout the paper, $\alpha$ denotes be width of the delta-like initial condition, $\delta$ denotes the parameter in the event \eqref{e.eventhh} and similar events, $N$ is the scaling parameter,
\begin{align}
	&A\lesssim B \text{ means } \limsup_{\delta\to 0} \limsup_{\alpha\to 0} \limsup_{N\to\infty} (N^3T)^{-1}\log(A/B) \leq 0,
\\
	&A\ll B \text{ means } \limsup_{\delta\to 0} \limsup_{\alpha\to 0} \limsup_{N\to\infty} (N^3T)^{-1}\log(A/B) < 0,
\end{align}
and $A\sim B$ means $A\lesssim B$ and $B\lesssim A$.

\subsection*{Outline}
In Section~\ref{s.momt}, we recall the relevant results from \cite{tsai2023high} on the moments of the SHE and use them to obtain the $n$-point LDP.
In Section~\ref{s.shape}, we prove that the KPZ height function follows the limit shape at any given point in $(0,1]\times\R$, up to an exponentially small probability.
In Section~\ref{s.conti}, we develop certain continuity estimates, which turns the pointwise result in Section~\ref{s.shape} to a result that holds on any compact subset of $(0,1]\times\R$, whereby completing the proof of Theorem~\ref{t.main}.
For technical reasons,  we will consider $\ZZ_{N,\alpha}$ and $\hh_{N,\alpha}$ instead of $\ZZ_{N}$ and $\hh_{N}$ in Sections~\ref{s.momt}--\ref{s.shape}.
Then, in Section~\ref{s.conti}, we use the continuity estimates to approximate the latter by the former.

\subsection*{Acknowledgments}
The work was initiated during ``PNW Integrable Probability Conference'' at Oregon State University.
We thank the conference organizer, Axel Saenz, for their hospitality.
We thank Sayan Das for helpful discussions.
The research of Tsai was partially supported by the NSF through DMS-2243112 and the Sloan Foundation Fellowship.

\section{Moments and $n$-point LDP}
\label{s.momt}
Here we introduce some properties of the moments of the SHE and use them to obtain the LDP.

The relevant initial condition for the SHE in Theorem~\ref{t.main} is the delta initial condition, but for technical reasons we will first consider the delta-like initial condition.
Namely, we will work with $\ZZ_{N,\alpha}$ and $\hh_{N,\alpha}$ hereafter until the end of Section~\ref{s.shape}.
Later, in Section~\ref{s.conti.tmain}, we will argue that they well approximate $\ZZ_{N}$ and $\hh_{N}$.

\subsection{Results from \cite{tsai2023high}}
\label{s.momt.tsai2023high}
Let us recall the relevant results from \cite{tsai2023high}.
For $\vecxx=(\xx_1<\ldots<\xx_n)$ and $\vecmm = (\mm_\cc)_{\cc=1}^n \in [0,\infty)^n$, the moment Lyapunov exponents of interest is the limit of 
\begin{align}
	\label{e.momentL}
	\frac{1}{N^3T} \log \EE\Big[ \prod_{\cc=1}^n Z_{N,\alpha}(t_0,\xxi; t,\xx_\cc)^{N\mm_\cc} \Big],
\end{align}
with $N\to\infty$ first (with $T=T_N$ and under \eqref{e.scaling}) and $\alpha\to 0$ later.
The results from \cite{tsai2023high} are for integer moments, so an integer part is implicitly taken whenever needed; for example, $N\mm_\cc := \lceil N\mm_\cc\rceil$ in \eqref{e.momentL}.
As will be seen later, the limit of \eqref{e.momentL} is continuous in $\vecmm\in[0,\infty)^n$, so obtaining the limit for $\lceil N\mm_\cc\rceil\in(\Z_{> 0})^n$ automatically gives the limit for $ N\mm_\cc\in [0,\infty)^n$.

Describing the limit of \eqref{e.momentL} requires considering measure-valued functions.
Set $\mm=\mm_1+\ldots+\mm_n$ and let $(\mm\Psp(\R))$ denote the space of positive Borel measures on $\R$ with total mass $\mm$.
For $\lambda\in\mm\Psp(\R)$, write $\int_\R\lambda(\d x)\, \f(x)=:\ip{\lambda,\f}$ and let $\quant[\lambda](a):=\inf\{x\in\R: \ip{\lambda,\ind_{(-\infty,x]}}\geq a\}$ denote the quantile function or inverse CDF.
Endow the space $\mm\Psp(\R)$ with the weak* topology, and, for $ \mu\in\Csp([s',s''],\mm\Psp(\R)) $, set
\begin{align}
	\label{e.mom}
	\mom_{[s',s'']}(\mu)
	:=
	\int_{s'}^{s''} \d s\, \Big(\sum_{x} \frac{1}{24} \ip{ \mu, \ind_{\{x\}} }^3 - \int_0^{\mm} \d a \, \frac12 \big( \partial_s \quant[\mu] \big)^2\Big).
\end{align}
Hereafter, we adopt shorthand notation such as $\mu:=\mu(s)$ and $\quant[\mu]:=\quant[\mu(s)](a)$.
In \eqref{e.mom}, the sum runs over all atoms of $\mu(s)$, and the integral is interpreted as $\infty$ when $\partial_s\quant[\mu]\notin\Lsp^2([s',s'']\times[0,\mm])$. 
Let us briefly explain how the above objects are related to \eqref{e.momentL}. 
With the aid of the Feynman--Kac formula, one can express the moment in \eqref{e.momentL} in terms of the exponential moments of the localtimes of $N\mm$ independent BMs.
The limit of \eqref{e.momentL} is then controlled by the large deviations of these BMs and their localtimes.
An element $ \mu\in\Csp([s',s''],\mm\Psp(\R)) $ describes a deviation of the empirical measure of the BMs.
Note that $\mu=\mu(s)$ uses the \emph{backward} time $s$ since the Feynman--Kac formula involves a time reversal; see~\eqref{e.feynmankac}.
The quantity~\eqref{e.mom} encodes the contribution of the localtimes (in the first term in the integral) and of the sample-path LDP for the BMs (in the second term). 

The limit of \eqref{e.momentL} is described by
\begin{align}
	\label{e.momL}
	\momshe\big(\xxi\xrightarrow{\termt} (\vecxx,\vecmm)\big)
	:=
	\sup\Big\{ \mom_{[0,\termts]}(\mu) \ : \ \mu(0)=\sum_{\cc=1}^n \mm_\cc \delta_{\xx_\cc}, \ \  \mu(\termts)=\mm\delta_{\xxi} \Big\},
\end{align}
where $\mm=\mm_1+\ldots+\mm_n$.
In light of the time reversal mentioned in the last paragraph, $\mu(0)=\sum_{\cc=1}^n \mm_\cc \delta_{\xx_\cc}$ corresponds to the points probed in \eqref{e.momentL} at time $t=\termt+t_0$, while $\mu(\termt)=\mm\delta_{\xxi}$ corresponds to the initial condition of $\ZZ_{N,\alpha}$ at time $t_0$.

The following properties are proven in \cite[Cor.\ 2.4 and Thm.\ 2.5]{tsai2023high}.
Recall $\lesssim$, $\ll$, and $\sim$ from Section~\ref{s.intro.notation}.
\begin{enumerate}[leftmargin=20pt, label=(\Alph*)]
\item \label{property.limit}
$
	\eqref{e.momentL} \sim \momshe\big(\xxi\xrightarrow{t-t_0} (\vecxx,\vecmm)\big)
$
\item \label{property.legendre}
For fixed $\vecxx$ and $t$, view $\ratekpz(t,\vecxx,\vechv)=:\ratekpz(\vechv)$ as a function on $\hvspacec(t,\vecxx)$ and view $\momshe(0\xrightarrow{t} (\vecxx,\vecmm))=:\momshe(\vecmm)$ as a function on $[0,\infty)^n$.
These functions are continuous, strictly convex, and the Legendre transform of each other.
Further, $\nabla\ratekpz=\nabla_{\vechv}\ratekpz:\hvspacec(t,\vecxx)\to[0,\infty)^n$ is a homeomorphism, so the following relation characterizes Legendre dual variables $(\vechv,\vecmm)$:
\begin{align}
	\label{e.legendre.dual}
	(\nabla_{\vechv}\ratekpz)(t,\vecxx,\vechv) = (\nabla\ratekpz)(\vechv) = \vecmm.
\end{align}
\item[\mylabel{property.intermediate.}{\ref*{property.intermediate}'}]
Let us prepare the notation for Property~\ref{property.intermediate}.
Fix $\vecxx\in\R^n$, and consider a Legendre pair $(\vechv,\vecmm)$ as in \eqref{e.legendre.dual} at time $1$, namely $(\nabla_{\vechv}\ratekpz)(1,\vecxxi,\vechvi)=\vecmmi$.
Take any intermediate time $\intermt\in(0,1]$.
Evolving in backward time, some of the noise corridors may have merged by time $s=1-\intermt$.
Let $\{\optimal_\cc(1-\intermt)\}_{\cc=1}^n = \{ \xxi_{1}<\ldots<\xxi_{n'} \}$ denote the distinct positions of the clusters at that time.
Accordingly, let $\vechvi := (\mps(\intermt,\xxi_{\aa}))_{\aa=1}^{n'}$, $\Aranch(\aa):=\{\cc:\optimal_\cc(1-\intermt)=\xxi_\aa\}$, $\mmi_\aa:=\sum_{\cc\in\Aranch(\aa)} \mm_{\cc}$, and $\vecmmi:=(\mmi_\aa)_{\aa=1}^{n'}$; see Figure~\ref{f.intermt} for an illustration.
\item 
\label{property.intermediate}
Notation as in \ref{property.intermediate.}. 
\begin{align}
	\label{e.legendre.dual.intermediate}
	(\nabla_{\vechv}\ratekpz)(\intermt,\vecxxi,\vechvi) & = \vecmmi,
\\
	\label{e.treesg}
	\momshe\big( 0 \xrightarrow{1} (\vecxx,\vecmm) \big)
	&=
	\momshe\big( 0 \xrightarrow{\intermt} (\vecxxi,\vecmmi) \big)
	+
	\sum_{\aa=1}^{n'} \momshe\big( \xxi_\aa \xrightarrow{1-\intermt} (\xx_\cc,\mm_\cc)_{\cc\in\Aranch(\aa)} \big).
\end{align}
\item \label{property.loc}
Recall $\ZZ_{N,\alpha}$ and $\ZZ^{\betaloc}_{N,\alpha}$ from \eqref{e.ZZ.omit.initial}.
For any nonempty $A\subset\{1,\ldots,n\}$ and $\beta>0$,
\begin{align*}
	\EE\Big[ 
		\prod_{\cc\in A} \big(\ZZ_{N,\alpha}-\ZZ^{\betaloc}_{N,\alpha}\big)(1,\xx_\cc)^{N\mm_\cc} 
		\cdot \prod_{\cc\notin A} \ZZ_{N,\alpha}(1,\xx_\cc)^{N\mm_\cc} 
	\Big]
	\ll
	\exp\big( N^3T\momshe\big(0\xrightarrow{1} (\vecxx,\vecmm)\big)\big).
\end{align*}
\end{enumerate}

\subsection{From moments to $n$-point LDP}
\label{s.momt.ldp}
Properties~\ref{property.limit}--\ref{property.legendre} immediately give the $n$-point LDP for $\hh_{N,\alpha}$.
Recall $\hh_{N,\alpha}(t,x)$ from \eqref{e.ZZ.omit.initial} and let
\begin{align}
	\label{e.eventhh.}
	\eventhh_{N,\alpha,\delta}(t,\vecxx,\vechv)
	:=
	\big\{ |\hh_{N,\alpha}(t,\xx_\cc)-\hv_\cc|\leq \delta, \cc=1,\ldots,n \big\}.
\end{align}
Since $\ZZ_{N,\alpha}(t,x)=\exp(N^2T\hh_{N,\alpha}(t,x))$, Property~\ref{property.limit} gives the limit of the $n$-point log moment generating function of $\hh_{N,\alpha}(t,\Cdot)$.
Further, by Property~\ref{property.legendre}, the limit is strictly convex in $\vecmm\in\hvspacec(t,\vecxx)$.
Hence the $n$-point LDP of $\hh_{N,\alpha}(t,\Cdot)$ follows.
\begin{cor}\label{c.ldp.soften}
For any $t\in(0,1]$, $\vecxx=(\xx_1<\ldots<\xx_n)$, and $\vechv\in\hvspacec(t,\vecxx)^\circ$,
\begin{align}
	\label{e.c.ldp.soften}
	\PP[\eventhh_{N,r,\delta}(t,\vecxx,\vechv)] \sim \exp\big( -N^3T\ratekpz(t,\vecxx,\vechv)\big).
\end{align}
\end{cor}

Next, let us state an LDP upper bound that holds for all $\vechv\in\R^n$.
The bound will serve as a technical tool in Section~\ref{s.shape.corridors}, is suboptimal on $\R^{n}\setminus\hvspacec$, but suffices for our purpose.
Recall $\hf{}=\hf{,t,\vecxx,\vechv}$ from Section~\ref{s.intro.ratefn} and recall that $\parab(t,x):=-x^2/(2t)$.
Consider the convex-hull analog of it: Let $\hfc{}=\hfc{,t,\vecxx,\vechv}\in\Csp(\R)$ be the function whose hypograph $\{(x,r): r\leq \hfc{}(x),x\in\R\}$ is the convex hull of $\{(x,\parab(t,x))\}_{x\in\R}\cup\{(\xx_\cc,\hv_\cc)\}_{\cc=1}^n$.
Indeed, when $\hv\in\hvspacec(t,\vecxx)$, $\hf{,t,\vecxx,\vechv}=\hfc{,t,\vecxx,\vechv}$, but they differ when $\hv\notin\hvspacec(t,\vecxx)$.
Set
\begin{align}
	\label{e.ratekpz.extended}
	\big(\ratekpz\circ\hfc{}\big)(t,\vecxx,\vechv)
	:=
	\ratekpz\big(t,\vecxx,(\hfc{,t,\vecxx,\vechv}(t,\xx_\cc))_{\cc=1}^n\big).
\end{align}
\begin{lem}\label{l.ldp.upb}
Fix $t$ and $\vecxx$.
\begin{enumerate}[leftmargin=20pt, label=(\alph*)]
\item \label{l.ldp.upb.1}
Take any Legendre-dual variables $(\vechv,\vecmm)$ in $\hvspacec(t,\vecxx)^\circ\times(0,\infty)^n$.
The map $\vec{r}\mapsto(\vec{r}\cdot\vecmm - (\ratekpz\circ\hfc{})(t,\vecxx,\vec{r}))$ is concave and continuous on $\R^n$ and has a unique maximum at $\vec{r}=\vechv$.
\item \label{l.ldp.upb.2}
For any closed $F\subset\R^n$,
\begin{align}
	\label{e.l.ldp.upb.2}
	\PP\big[ (\hh_{N,\alpha}(t,\vecxx_\cc))_{\cc=1}^n \in F \big]
	\lesssim
	\exp\Big( -N^3T \, \inf_F \big(\ratekpz\circ\hfc{}\big)\Big).
\end{align}
\end{enumerate}
\end{lem}
\noindent{}
Lemma~\ref{l.ldp.upb} is proven in Appendix~\ref{s.a.tools}.

\section{Limit shape at one point}
\label{s.shape}

Given the $n$-point LDP in Corollary~\ref{c.ldp.soften}, we move on to the limit shape.
Hereafter, we fix $\vecxx$ and $\vechv\in\hvspacec(1,\vecxx)^\circ$.
Recall $\eventhh_{N,\alpha,\delta}(1,\vecxx,\vechv)$ from \eqref{e.eventhh.}.
With $\vecxx$ and $\vechv$ being fixed, hereafter we write $\eventhh_{N,\alpha,\delta}(1,\vecxx,\vechv)=\eventhh_{N,\alpha,\delta}$. 
The goal of this section is to prove the proposition.
\begin{prop}\label{p.shape}
Notation as above.
For any $(t,x)\in(0,1]\times\R$ and $R\in(0,\infty)$,
\begin{align}
	\label{e.p.shape}
	\PP \big[ \big\{ |\hh_{N,\alpha}(t,x)-\mps(t,x)| \geq \tfrac{1}{R} \big\} \, \big| \, \eventhh_{N,\alpha,\delta} \big]
	\ll
	\exp\big( - N^3T \cdot 0 \big) = 1.
\end{align}
\end{prop}
%
\noindent{}%
We call Proposition~\ref{p.shape} a one-point result because the event involves one $(t,x)$.
By the union bound, the one-point result automatically extends to finitely many points in $(0,1]\times\R$.
Going from finitely many points to a compact subset of $(0,1]\times\R$ requires certain continuity estimates, which will be done in Section~\ref{s.conti}.

Actually, the proof of Proposition~\ref{p.shape} itself requires certain continuity estimates.
We state them here and defer the proof to Section~\ref{s.conti}.
Hereafter we write $c(v_1,v_2,\ldots)\in(0,\infty)$ for a generic constant that depends only on the designed variables $v_1,v_2,\ldots$.
\begin{lem}\label{l.conti.}
\begin{enumerate}[leftmargin=20pt, label=(\alph*)]
\item[]
\item
\label{l.conti..1}
Take any $(t,y)\in(0,1]\times\R$.
There exists $c=c(t,y)$ such that, for all $r\geq\beta^{1/9}$,
\begin{align}
	\PP\Big[  
		\sup_{x\in[-\beta+y,y+\beta]} |\hh_{N,\alpha}(t,x)-\hh_{N,\alpha}(t,y)| \geq r
	\Big]
	\lesssim
	e^{-\frac{1}{c}\,N^3T r^{3/2}\beta^{-1/6}}.
\end{align}
\item
\label{l.conti..2}
For every $(t_1,y_1),(t_2,y_2)\in[0,1]\times\R$ with $t_1<t_2$ and every $\beta,R\in(0,\infty)$,
\begin{align}
	\PP\Big[  
		\sup_{x\in[-\beta+y_1,y_1+\beta]} \Big| \frac{1}{N^2T} \log \big( \ZZ_{N}(t_1,x;t_2,y_2) \sqrt{T}\big) + \frac{(y_2-x)^2}{2(t_2-t_1)} \Big| \geq \frac{1}{R}
	\Big]
	\ll
	e^{-N^3T \cdot 0}.
\end{align}
\end{enumerate}
\end{lem}
\noindent%
To understand where the factor $(y_2-x)^2/(2(t_2-t_1))$ in Lemma~\ref{l.conti.}\ref{l.conti..2} came from, refer to the Feynman--Kac formula~\eqref{e.ZZ.ptp}. 
There, if we replace $\noise$ with zero, the result becomes the heat kernel, which reads $\hk(T(t_2-t_1),NT(y_2-x))$ here.
Applying $(N^2T)^{-1}\log(\Cdot\sqrt{T})$ to the heat kernel gives $-(y_2-x)^2/(2(t_2-t_1))$ plus a negligibly term.
Put it differently, Lemma~\ref{l.conti.}\ref{l.conti..2} says that $\ZZ_{N}(t_1,x;t_2,y_2)$ \emph{typically} behaves as if the spacetime white noise did not contribute at all.

\subsection{Limit shape along the noise corridors}
\label{s.shape.corridors}
Here we prove Proposition~\ref{p.shape} when $(t,x)$ lies on a noise corridor.
More precisely, letting $\vecxx$, $\vechv$, $\vecmm$, $\intermt$, $\vecxxi$, $\vechvi$, $\vecmmi$, and $\Aranch(\aa)$ be as in \ref{property.intermediate.} in Section~\ref{s.momt.tsai2023high}, we assume $(t,x)=(\intermt,\xxi_{\aa_0})$ for some $\aa_0$.
We write $t$ as $\intermt$ to be consistent with the notation in \ref{property.intermediate.}.

The first step is to perform localization.
Let $\eventhhi_{N,\alpha}:=\{|\hh_{N,\alpha}(\intermt,\xxi_{\aa_0})-\mps(\intermt,\xxi_{\aa_0})|>1/R\}$ be the first event in \eqref{e.p.shape}.
Recall that the event $\eventhh_{N,\alpha,\delta}$ controls the value of $\hh_N(1,\xx_\cc)$, $\cc=1,\ldots,n$, and use this fact to write
$
	\PP[\eventhhi_{N,\alpha}\cap\eventhh_{N,\alpha,\delta}]
	\lesssim
	\exp(-N^3T\vecmm\cdot\vechv)
	\cdot
	\EE[ \prod_{\cc=1}^n \ZZ_{N,\alpha}(1,\xx_\cc)^{N\mm_\cc} \ind_{\eventhhi_{N,\alpha}} ].
$
Within the last expectation, write $\ZZ_{N,\alpha}$ as the sum of $\ZZ^{\betaloc}_{N,\alpha}$ and $(\ZZ_{N,\alpha}-\ZZ^{\betaloc}_{N,\alpha})$, note that the both are non-negative, and use the inequality $(A+B)^{N\mm_\cc}\leq 2^{N\mm_\cc}A^{N\mm_\cc}+2^{N\mm_\cc}B^{N\mm_\cc}$.
The result gives
\begin{align}
	\PP[\eventhhi_{N,\alpha}\cap\eventhh_{N,\alpha,\delta}]
	\lesssim
	2^{N\mm} \, e^{-N^3T\vecmm\cdot\vechv}\,
	\Big( 
		\EE\Big[ \prod_{\cc=1}^n \ZZ^{\betaloc}_{N,\alpha}(1,\xx_\cc)^{N\mm_\cc} \ind_{\eventhhi_{N,\alpha}} \Big]
		+
		(\text{remainder})
	\Big),
\end{align}
where the remainder is given by summing the expectation in Property~\ref{property.loc} over all nonempty $A\subset\{1,\ldots,n\}$.
Under \eqref{e.scaling}, the factor $2^{N\mm}$ is negligible, so we absorb it into $\lesssim$.
By Property~\ref{property.loc}, the remainder is $\lesssim \exp(N^3T(\momshe(0\xrightarrow{\scriptscriptstyle 1} (\vecxx,\vecmm))-f_1(\beta)))$, for some $f_1:(0,\infty)\to(0,\infty)$.
Recall that $\vecmm$ was chosen to be the Legendre-dual variable of $\vechv$.
Multiply both sides of the last bound by the factor $ \exp(-N^3T\vecmm\cdot\vechv) $ and recognize that $(\vecmm\cdot\vechv-\momshe(0\xrightarrow{\scriptscriptstyle 1} (\vecxx,\vecmm)))$ is $\ratekpz(1,\vecxx,\vechv)$ with the aid of Property~\ref{property.legendre}.
Doing so gives
\begin{align}
	\label{e.p.shape.localization}
	\PP[\eventhhi_{N,\alpha}\cap\eventhh_{N,\alpha,\delta}]
	\lesssim
	e^{-N^3T\vecmm\cdot\vechv}\,\EE\Big[ \prod_{\cc=1}^n \ZZ^{\betaloc}_{N,\alpha}(1,\xx_\cc)^{N\mm_\cc} \ind_{\eventhhi_{N,\alpha}} \Big]
	+
	e^{-N^3T(\ratekpz(1,\vecxx,\vechv)+f_1(\beta))}.
\end{align}

Next, we decompose the expectation in \eqref{e.p.shape.localization} at time $\intermt$.
Recall $\ZZ^{\betaloc}_{N,\alpha}(t,x)$ and $\ZZ^{\betaloc}_{N,\alpha}(t_0,x',t,x)$ from Section~\ref{s.intro.notation} and invoke the semigroup identity
\begin{align}
	\label{e.cutZZ}
	\ZZ^{\betaloc}_{N,\alpha}(1,\xx_\cc)
	=
	\int NT\d x\  \ZZ^{\betaloc}_{N,\alpha}(\intermt,x) \ \ZZ^{\betaloc}_{N}(\intermt,x;1,\xx_\cc).
\end{align}
Take any $\cc\in\Aranch(\aa)$ and $\cc'\in\Aranch(\aa')$ with $\aa\neq\aa'$.
By the definition of $\Aranch(\Cdot)$, $\inf_{1-s\in[\intermt,1]}|\optimal_{\cc}(s)-\optimal_{\cc'}(s)|>0$.
Take $\beta$ small enough such that $\inf_{1-s\in[\intermt,1]}|\optimal_{\cc}(s)-\optimal_{\cc'}(s)|>2\beta$, for all such $\cc,\cc'$.
Under this condition, the sets of random variables
$
	\{ \ind_{\eventhhi_{N,\alpha}}, \ZZ^{\betaloc}_{N,\alpha}(\intermt,x) \}
$
and
$	\{
		\ZZ^{\betaloc}_{N}(\intermt,x;1,\xx_\cc)
	\}_{\cc\in\Aranch(1)}
$
\ldots and
$	
	\{
		\ZZ^{\betaloc}_{N}(\intermt,x;1,\xx_\cc)
	\}_{\cc\in\Aranch(n')}	
$
are independent, because they are measurable with respect to the noise over disjoint regions in spacetime.
Using this property to evaluate the moment gives
\begin{align}
	\label{e.p.shape.cut}
	\text{(expectation in \eqref{e.p.shape.localization})}
	=
	\Big( \prod_{\aa=1}^{n'} \prod_{\cc\in\Aranch(\aa)} \prod_{\ii=1}^{N\mm_\cc} \int_{-\beta+\xxi_\aa}^{\xxi_\aa+\beta} \hspace{-5pt} NT \d x_{\cc,\ii} \Big)
	\cdot
	E_0 \cdot \prod_{\aa=1}^{n'} E_\aa,
\end{align}
where
$
	E_0 := \EE [\prod_{\aa=1}^{n'} \prod_{\ii=1}^{N\mmi_\aa} \ZZ^{\betaloc}_{N,\alpha}(\intermt,x_{\aa,\ii}) \cdot \ind_{\eventhhi_{N,\alpha}}]
$
and
$
	E_\aa := \EE [\prod_{\cc\in\Aranch(\aa)} \prod_{\ii=1}^{N\mm_\cc} \ZZ^{\betaloc}_{N}(\intermt,x_{\cc,\ii};1,\xx_\cc)],
$
and the bounds of the integrals come from the constraints on the BMs in $\ZZ^{\betaloc}_{N}(\intermt,x_{\cc,\ii};1,\xx_\cc)$.

Next we bound $E_0$.
First, write $\ZZ^{\betaloc}_{N,\alpha}(t,x)\leq \ZZ_{N,\alpha}(t,x)=\exp(N^2T\hh_{N,\alpha}(t,x))$.
In light of \eqref{e.p.shape.cut}, we consider $x\in[-\beta+\xxi_\aa,\xxi_\aa+\beta]$ only.
Set $A_{N,\alpha}:=\max_{\aa=1}^{n'} \sup_{x\in[-\beta+\xxi_\aa,\xxi_\aa+\beta]} |\hh_{N,\alpha}(\intermt,x)-\hh_{N,\alpha}(\intermt,\xxi_\aa)|$, which controls the error when one approximates the last exponential by its value at $x=\xxi_\aa$.
We have
$
	E_0 \leq \EE[ \exp(N^3T(\sum_{\aa}\mmi_{\aa}\hh_{N,\alpha}(\intermt,\xxi_\aa)))\exp(N^3T\mm A_{N,\alpha})\ind_{\eventhhi_{N,\alpha}} ].
$
Invoke a small parameter $\gamma>0$ and decompose the last expectation into $\{A_{N,\alpha}\leq \gamma\}$ and $\{A_{N,\alpha}> \gamma\}$:
\begin{align}
   E_{0,1} &:= \EE\big[ e^{N^3T\sum_{\aa}\mmi_{\aa}\hh_{N,\alpha}(\intermt,\xxi_\aa)} e^{N^3T\mm A_{N,\alpha}}\ind_{\eventhhi_{N,\alpha}} \ind_{\{A_{N,\alpha}\leq \gamma\}} \big],
\\
   E_{0,2} &:= \EE\big[ e^{N^3T\sum_{\aa}\mmi_{\aa}\hh_{N,\alpha}(\intermt,\xxi_\aa)} e^{N^3T\mm A_{N,\alpha}}\ind_{\eventhhi_{N,\alpha}} \ind_{\{A_{N,\alpha}> \gamma\}} ].
\end{align}
In $E_{0,1}$, bound the second exponential by $\exp(N^3T\mm\gamma)$ and apply Lemma~\ref{l.ldp.upb}\ref{l.ldp.upb.2}.
Doing so gives $ E_{0,1} \lesssim \exp(N^3T (\mm\gamma+\sup\{ \vec{r}\cdot\vecmmi - (\ratekpz\circ\hfc{})(\intermt,\vecxxi,\vec{r}) \}))$, where the supremum runs over all $\vec{r}\in\R^{n'}$ with the constraint $|r_{\aa_0}-\hvi_{\aa_0}|\geq1/R$.
Had it been without the constraint, by Lemma~\ref{l.ldp.upb}\ref{l.ldp.upb.1} and \eqref{e.legendre.dual.intermediate}, the supremum would be achieved uniquely at $\vec{r}=\vechvi$, and be equal to $\momshe(0\xrightarrow{\scriptscriptstyle \intermt}(\vecxxi,\vecmmi))$.
\emph{With} the constraint, the supremum is strictly smaller than that, so $ E_{0,1} \lesssim \exp(N^3T (\mm\gamma+\momshe(0\xrightarrow{\scriptscriptstyle \intermt}(\vecxxi,\vecmmi))-1/c_2))$.
As for $E_{0,2}$, forgo the indicator $\ind_{\eventhhi_{N,\alpha}}$ and apply the Cauchy--Schwarz inequality to separate the two exponentials
\begin{align}
   E_{0,2} 
   \leq 
   \big( \EE\big[ e^{2N^3T\sum_{\aa}\mmi_{\aa}\hh_{N,\alpha}(\intermt,\xxi_\aa)}] \big)^{1/2} 
   \big( \EE\big[e^{2N^3T\mm A_{N,\alpha}} \ind_{\{A_{N,\alpha}> \gamma\}} \big]\big)^{1/2}.
\end{align}
By Property~\ref{property.limit}, the first expectation on the right hand side is $\lesssim \exp(c_3N^3T)$.
By Lemma~\ref{l.conti.}\ref{l.conti..1}, the second expectation on the right hand side is $\lesssim \exp(-N^3Tf_4(\beta,\gamma))$ for some $f_4:(0,\infty)^2\to(0,\infty)$ such that $\lim_{\beta\to 0}f_4(\beta,\gamma)=\infty$ for each $\gamma>0$.
Altogether, 
\begin{align}
	\label{e.p.shape.E0}
	E_0
	\lesssim
	\exp\big( N^3T \,\big(\mm\gamma+\momshe(0\xrightarrow{\intermt}(\vecxxi,\vecmmi))-1/c_2\big) \big)
	+
	\exp\big( \tfrac{1}{2}N^3T \big(c_3-f_4(\beta,\gamma)\big)\big).
\end{align}

Let us complete the proof.
Insert the bound \eqref{e.p.shape.E0} into \eqref{e.p.shape.cut}, observe that the bound does not depend on the integration variables, and factor it out of the integrals.
Next, rewrite the remaining integrals as 
$ 
	\prod_{\aa=1}^{n'} \EE[\prod_{\cc\in\Aranch(\aa)} \big(\int_{-\beta+\xxi_\aa}^{\xxi_\aa+\beta} NT \d x_{\cc} \, \ZZ^{\betaloc}_{N}(\intermt,x_{\cc,\ii};1,\xx_\cc))^{N\mm_\cc}],
$
recognize the last integral as $\ZZ^{\betaloc}_{N,\beta}(\intermt,\xx_\aa;t,\xx_\cc)$, and bound it by $\ZZ_{N,\beta}(\intermt,\xxi_\aa;1,\xx_\cc)$.
Doing so gives
\begin{align}
	\label{e.temp1}
	\prod_{\aa=1}^{n'} \EE\Big[\prod_{\cc\in\Aranch(\aa)} \ZZ_{N,\beta}(\intermt,\xxi_\aa;1,\xx_\cc)^{N\mm_\cc}\Big].
\end{align}
Applying Property~\ref{property.limit} to each of the last expectations gives 
\begin{align}
	\eqref{e.temp1}
	\lesssim 
	\exp\Big(N^3T\Big(\sum_{\aa=1}^n \momshe(\xxi_\aa\xrightarrow{1-\intermt}(\xx_\cc,\mm_\cc)_{\cc\in\Aranch(\aa)})+f_5(\beta)\Big)\Big),
\end{align}
where $f_5(\beta)\to 0$ as $\beta\to 0$.
Combine the preceding results with the aid of \eqref{e.treesg}.
We arrive at
\begin{subequations}
\label{e.p.shape.cut.}
\begin{align}
	\text{(the expectation in \eqref{e.p.shape.localization})}
	\lesssim&
	\exp\big(N^3T \big( \momshe\big(0\xrightarrow{1} (\vecxx,\vecmm)\big) +\mm\gamma-1/c_2 + f_5(\beta) \big)\big)
\\
	&+
	\exp\big( \tfrac{1}{2}N^3T \big(c'_3-f_4(\beta,\gamma)\big)\big).
\end{align}
\end{subequations}
Set $\gamma:=1/(2\mm c_2)$ and fix a small enough $\beta>0$ such that $-1/(2c_2)+f_5(\beta)<0$ and that $c'_3-f_4(\beta,1/(2\mm c_2))< -\momshe(0\xrightarrow{\scriptscriptstyle 1} (\vecxx,\vecmm))$.
Then, insert \eqref{e.p.shape.cut.} into \eqref{e.p.shape.localization} and recognize $(-\vechv\cdot\vecmm+\momshe(0\xrightarrow{\scriptscriptstyle 1} (\vecxx,\vecmm)))$ as $-\ratekpz(1,\vecxx,\vecmm)$.
Doing so gives
\begin{align}
	\PP[\eventhhi_{N,\alpha}\cap\eventhh_{N,\alpha,\delta}]
	=
	\PP\big[ \big\{ \big|\hh_{N,\alpha}(\intermt,\xxi_{\aa_0})-\mps(\intermt,\xxi_{\aa_0})\big|>\tfrac{1}{R} \big\} \cap\eventhh_{N,\alpha,\delta} \big]
	\ll
	e^{-N^3T\ratekpz(1,\vecxx,\vecmm)}.
\end{align}
By Corollary~\ref{c.ldp.soften}, the right hand side is $\sim \PP[\eventhh_{N,\alpha,\delta}]$, so the desired result follows.

\subsection{Limit shape outside the noise corridors}
\label{s.shape.outcorridors}
To complete the proof of Proposition~\ref{p.shape}, we now consider any $(t,x)$ outside the noise corridors.

The key is to examine the increments of $\mps$ and of $\hh_{N,\alpha}$ along the characteristics.
Recall the characteristics of $\mps$ from Section~\ref{s.intro.ratefn}, take any $(t_1,y_1),(t_2,y_2)\in(0,1]\times\R$, with $t_1<t_2$, that are connected by a characteristic, and let $\mathsf{Chara}_{12}$ denote the open segment of the characteristic between $(t_1,y_1)$ and $(t_2,y_2)$, excluding the endpoints.
We assume that $\mathsf{Chara}_{12}$ does not intersect any noise corridor.

To analyze the increment of $\mps$ along $\mathsf{chara}_{12}$, note that $\mps$ solves \eqref{e.iburgers} \emph{classically} away from the noise corridors, so 
$
	\frac{\d~}{\d t} \mps|_{\mathsf{Chara}_{12}} = (\partial_t \mps + \partial_x \mps \cdot v)|_{\mathsf{Chara}_{12}} = (\frac12 (\partial_x \mps)^2 + \partial_x \mps \cdot v)|_{\mathsf{Chara}_{12}},
$
where $v:=(x_2-x_1)/(t_2-t_1)$ is the velocity when one travels along the characteristics in forward time.
Also, $(\partial_x\mps)|_{\mathsf{Chara}_{12}}=-v$, so the differentiation evaluates to $-v^2/2$, whereby
\begin{align}
	\label{e.increment.mps}
	\mps(t_2,x_2) - \mps(t_1,x_1) = - (x_2-x_1)^2/(2(t_2-t_1)).
\end{align}

Next, let us analyze the increment of $\hh_{N,\alpha}$.
Hereafter, we say an event $\mathcal{A}$ holds \textbf{up to an $\eventhh$-small probability} if $\PP[\mathcal{A}^\comple] \ll \exp(-N^3T\ratekpz(1,\vecxx,\vechv))$, which implies $\PP[\mathcal{A}^\comple|\eventhh_{N,\alpha,\delta}]\ll\exp(-N^3T\cdot 0)$.
We only need a lower bound on the increment.
Write $\exp(N^2T\hh_{N,\alpha}(t_2,y_2))=\ZZ_{N,\alpha}(t_2,y_2)$, assume $t_1>0$, invoke the semigroup identity
$
	\ZZ_{N,\alpha}(t_2,y_2)
	=
	\int_{\R} NT \d x \, \ZZ_{N,\alpha}(t_1,x) \, \ZZ_{N}(t_1,x;t_2,y_2),
$
take a small parameter $\beta>0$, forgo the integral outside $[-\beta+y_1,y_1+\beta]$, and write the result as
\begin{align}
	\label{e.toward.increment.hh.}
	\exp(N^2T\hh_{N,\alpha}(t_2,y_2))
	\geq
	\int_{|x-y_1|\leq \beta} \d x \, \exp(N^2T\hh_{N,\alpha}(t_1,x)) \, NT\,\ZZ_{N}(t_1,x;t_2,y_2).
\end{align}
Recall $R$ from \eqref{e.p.shape} and apply Lemma~\ref{l.conti.}\ref{l.conti..1} with a small enough $\beta$ such that, up to an $\eventhh$-small probability,
$
	\inf_{[-\beta+y_1,y_1+\beta]}\hh_{N,\alpha}(t_1,x)\geq\hh_{N,\alpha}(t_1,y_1)-1/(4R).
$ 
Insert this inequality into \eqref{e.toward.increment.hh.} and take the logarithm.
Doing so gives, up to an $\eventhh$-small probability,
\begin{align}
	\label{e.toward.increment.hh}
	\hh_{N,\alpha}(t_2,y_2)-\hh_{N,\alpha}(t_1,y_1)+\frac{1}{4R} 
	\geq 
	\frac{1}{N^2T}\log\int_{|x-y_1|\leq \beta} \d x \, NT\,\ZZ_{N}(t_1,x;t_2,y_2).
\end{align}

Next, we seek to bound the last term in \eqref{e.toward.increment.hh} from below by Lemma~\ref{l.conti.}\ref{l.conti..2}.
As explained after that lemma, it concerns the \emph{typical} behaviors of $\ZZ_{N}(t_1,x;t_2,y_2)$.
In particular, the bound in Lemma~\ref{l.conti.}\ref{l.conti..2} is only $\ll\exp(-N^3T\cdot 0)$, not necessarily $\eventhh$-small.
To leverage Lemma~\ref{l.conti.}\ref{l.conti..2} into a result that holds up to an $\eventhh$-small probability, we invoke the following inequality, which is proven in Appendix~\ref{s.a.tools},
\begin{align}
	\label{e.fkg}
	\EE\Big[ &\ind_{\eventinc_{N}^\comple} \, \prod_{\cc=1}^{n} \ZZ_{N,\alpha}(1,\xx_{\cc})^{N\mm_\cc} \Big]
	\leq
	\PP\big[ \eventinc_{N}^\comple \big]
	\,\EE\Big[ \prod_{\cc=1}^{n} \ZZ_{N,\alpha}(1,\xx_{\cc})^{N\mm_\cc} \Big],
\\
	\label{e.fkg.event}
	&\eventinc_{N} := \Big\{\inf_{x\in[-\beta+y_1,y_1+\beta]}\Big(\frac{1}{N^2T}\log\big(\ZZ_N(t_1,x;t_2,y_2)\sqrt{T}\big)+\frac{(y_2-x)^2}{2(t_2-t_1)}\Big) \geq -\frac{1}{4R} \Big\}.
\end{align}
Consider $\PP[\eventinc_{N}^\comple\cap\eventhh_{N,\alpha,\delta}]$, recall that $\eventhh_{N,\alpha,\delta}$ controls the value of $\hh_{N,\alpha}(1,\xx_\cc)$, and use this property to bound the last probability by $\lesssim \exp(-N^3T\vechv\cdot\vecmm)\cdot(\text{left hand side of \eqref{e.fkg}})$.
Then, bound the two terms on the right hand side of \eqref{e.fkg} by Lemma~\ref{l.conti.}\ref{l.conti..2} and Property~\ref{property.limit}, respectively.
Doing so gives 
$
	\PP[\eventinc_{N}^\comple\cap\eventhh_{N,\alpha,\delta}]
	\ll
	\exp(N^3T(-\vechv\cdot\vecmm+\momshe(0\xrightarrow{\scriptscriptstyle 1} (\vecxx,\vecmm))).
$
Recognizing the last expression as $\exp(-N^3T\ratekpz(1,\vecxx,\vechv))$, we conclude that $(\eventinc_{N}\cap\eventhh_{N,\alpha,\delta})$ holds up to an $\eventhh$-small probability.
This result holds for all $t_1<t_2\in[0,1]$, including $t_1=0$, because Lemma~\ref{l.conti.}\ref{l.conti..2} and \eqref{e.fkg} do.

Let us now derive lower bounds on the increment of $\hh_{N,\alpha}$.
Consider first $t_1>0$, whence \eqref{e.toward.increment.hh} holds.
Apply the inequality in \eqref{e.fkg.event} to \eqref{e.toward.increment.hh}.
In the result, take $\beta$ to be small enough such that, for all $x\in[-\beta+y_1,y_1+\beta]$, the term $(y_2-x)^2/(2(t_2-t_1))$ approximates $(y_2-y_1)^2/(2(t_2-t_1))$ to within $1/R$.
Also, note that $(N^2T)^{-1}\log\sqrt{N^2T}\to 0$ under \eqref{e.scaling}.
Doing so gives, under the assumption that $t_1>0$ and up to an $\eventhh$-small probability,
\begin{align}
	\label{e.increment.hh}
	\hh_{N,\alpha}(t_2,y_2)-\hh_{N,\alpha}(t_1,y_1) \geq -(y_2-y_1)^2/(2(t_2-t_1)) - 3/(4R).
\end{align}
Next, consider $(t_1,y_1)=(0,0)$.
Write $\hh_{N,\alpha}(t_2,y_2):=(N^2T)^{-1}\log\int_{|x|\leq \alpha} \d x \, NT\,\ZZ_{N}(0,x;t_2,y_2)$, apply the inequality in \eqref{e.fkg.event} with $(t_1,y_1)=(0,0)$ to the last integral, and follow the same procedure to simplify the result.
The result gives, up to an $\eventhh$-small probability,
\begin{align}
	\label{e.increment.hh00}
	\hh_{N,\alpha}(t_2,y_2) \geq -y_2^2/(2t_2) - 3/(4R).
\end{align}

Let us prove the upper-half of \eqref{e.p.shape}: Up to an $\eventhh$-small probability, $\hh_{N,\alpha}(t,x) \leq  \mps(t,x) + 1/R$.
With $(t,x)$ being outside the noise corridors, there exists a unique $y\in\R$ such that $(t,x)$ and $(1,y)$ is connected by a characteristic; see Figure~\ref{f.shocks} and the last paragraph in Section~\ref{s.intro.ratefn}.
The points $(t,x)$ and $(1,y)$ satisfy the conditions on $(t_1,x_1)$ and $(t_2,x_2)$ in the second paragraph.
At time $1$, Corollary~\ref{c.ldp.soften} gives the multipoint LDP for $\hh_{N,\alpha}(1,\Cdot)$.
Apply Corollary~\ref{c.ldp.soften} at the $(n+1)$ points $\xx_1,\ldots,\xx_n,y$ to get
$
	\PP[ \{ |\hh_{N,\alpha}(1,y)- r|\leq\delta \}\cap \eventhh_{N,\alpha,\delta}] \sim \exp(-N^3T\ratekpz(1,\vecxx\cup\{y\},\vechv\cup\{r\})),
$
for all $r\in\R$ such that $\vechv\cup\{r\} \in \hvspacec(1,\vecxx\cup\{y\})^\circ$.
Referring back to \eqref{e.hvspacec} and Figure~\ref{f.shape.termt}, we see that the last condition is equivalent to $r >\hf{,1,\vecxx}(y)=\mps(1,y)$.
For all such $r$, as is readily checked from \eqref{e.ratekpz}, $\ratekpz(1,\vecxx\cup\{y\},\vechv\cup\{r\})<\ratekpz(1,\vecxx,\vechv)$.
Hence $ |\hh_{N,\alpha}(1,y)- r|\leq\delta \}\cap \eventhh_{N,\alpha,\delta} $ holds up to an $\eventhh$-small probability.
Since this holds for all $r >\mps(1,y)$ and since $\delta\to 0$, the event $ |\hh_{N,\alpha}(1,y)- r|\leq\delta \}\cap \eventhh_{N,\alpha,\delta} $ holds up to an $\eventhh$-small probability.
Combining this with \eqref{e.increment.hh} for $(t_1,y_1,t_2,y_2)=(t,x,1,y)$ and with \eqref{e.increment.mps} gives, up to an $\eventhh$-small probability, $\hh_{N,\alpha}(t,x) \leq  \mps(t,x) + 1/R$, which is the desired result.

Finally, we prove the lower-half of \eqref{e.p.shape}: Up to an $\eventhh$-small probability, $\hh_{N,\alpha}(t,x) \geq  \mps(t,x) - 1/R$.
With $(t,x)$ being outside the noise corridors, there exists a unique $(\intermt,\xxi)$ along the noise corridors such that $(\intermt,\xxi)$ and $(t,x)$ are connected by a characteristic; see Figure~\ref{f.shocks} and the last paragraph in Section~\ref{s.intro.ratefn}.
The points $(\intermt,\xxi)$ and $(t,x)$ satisfy the conditions on $(t_1,x_1)$ and $(t_2,x_2)$ in the second paragraph.
When $\intermt>0$, the result of Section~\ref{s.shape.corridors} gives, up to an $\eventhh$-small probability,
$ \hh_{N,\alpha}(\intermt,\xxi) \leq \mps(\intermt,\xxi) + 3/(4R) $.
Combining this inequality with \eqref{e.increment.hh} and \eqref{e.increment.mps} for $(t_1,y_1,t_2,y_2)=(\intermt,\xxi,t,x)$ gives the desired result $\hh_{N,\alpha}(t,x) \geq  \mps(t,x) - 1/R$.
When $\intermt=0$, necessarily $\xxi=0$, and $\mps(0,0)=0$.
Combining \eqref{e.increment.hh00} and \eqref{e.increment.mps} for $(t_1,y_1,t_2,y_2)=(\intermt,\xxi,t,x)$ gives the desired result.

\section{Continuity estimates}
\label{s.conti}
The main task in this section is to prove the result.
Recall $\hh_N$ from Section~\ref{s.intro.notation}.
\begin{prop}\label{p.conti}
Given any $R<\infty$, there exists a $c=c(R)$ such that, for all $(t_1,x_1),(t_2,x_2)\in[1/R,1]\times[-R,R]$ and $u\geq c$,
\begin{align}
	\label{e.conti.holder}
	\PP\big[ |\hh_N(t_2,x_2)-\hh_N(t_1,x_1)| \geq (|t_2-t_1|^{1/13}+|x_2-x_1|^{1/7})\,u \big] \leq \exp(-\tfrac{1}{c}N^3Tu^{3/2}).
\end{align}
\end{prop}
\noindent{}%
After proving this, we will use it and Proposition~\ref{p.shape} to complete the proof of Theorem~\ref{t.main}.

\subsection{Proof of Proposition~\ref{p.conti}}
\label{s.conti.}
Let us outline the proof.
Fix $R<\infty$, assume $t_1\leq t_2$, write $c=c(R)$ to simplify notation, and call a probability \textbf{affordable} if it is $\leq c\,\exp(-N^3Tu^{3/2}/c)$.
As will be seen, the proof below works for all $u\geq c$, for some $c$.
We seek to show that, up to an affordable probability,
\begin{subequations}
\label{e.conti.goal}
\begin{align}
	\ZZ_N(t_1,x_1) 
	\leq 
	& \exp\big( cN^2T(|t_2-t_1|^{1/13}+|x_2-x_1|^{1/7})u \big) \ZZ_N(t_2,x_2) 
\\	
	&+ \exp\big( -N^2Tu / (c\,((t_2-t_1)^{6/13} + |x_2-x_1|^{6/7})) \big) \, \ZZ_N(t_1,x_1).
\end{align}
\end{subequations}
Once this is done, half of the desired bound $\hh_N(t_1,x_1) \leq (|t_2-t_1|^{1/13}+|x_2-x_1|^{1/7})cu + \hh_N(t_2,x_2) $ follows.
The other half can be proven similarly, which we omit.
To obtain \eqref{e.conti.goal}, we start from the semigroup identity
$
	\ZZ_{N}(t_1,x_1)
	=
	\int_{\R} NT \d x \, \ZZ_{N}(t_0,x) \, \ZZ_{N}(t_0,x;t_1,x_1)
$
and divide the integral into two, one within and one outside a closed interval.
We will bound the integral outside the closed interval by the last term in \eqref{e.conti.goal}.
For the integral within the close interval, we will argue that $\ZZ_{N}(t_0,x;t_1,x_1)$ is approximately (at the exponential scale) bounded by $\ZZ_{N}(t_0,x;t_2,x_2)$ and use the semigroup identity in reserve to obtain the second last term in \eqref{e.conti.goal}.

We need some tools.
Let $\hk(t,x):=\exp(-x^2/(2t))/\sqrt{2\pi t}$ denote the heat kernel and set
\begin{align}
	\WW_N(t,x;t',x'):=\ZZ_N(t,x;t',x')/\hk(T(t'-t),NT(x'-x)).
\end{align}
Namely, $\WW_N$ is given by the Feynman--Kac formula \eqref{e.ZZ.ptp} with $\Ebm[\ldots\delta_{NTx'}(\X(T(t-t_0)))]$ replaced by the law of a Brownian bridge that starts from $NTx'$ and ends at $NTx$.
The Feynman--Kac formula shows that $W_N(0,0;\Cdot,\Cdot)$ and $W_N(t_0,x_0;\Cdot+t_0,\Cdot+x'_0)$ have the same law --- which we call the \textbf{shift invariance} --- and that $W_N(0,0;t_0,\Cdot)$ and $W_N(0,\Cdot;t_0,0)$ have the same law --- which we call the \textbf{time-reversal symmetry}.
The following bounds can be derived from existing one-point tail bounds of the KPZ equation through the line-ensemble arguments in \cite{corwin2021kpz}.
For the sake of completeness, we present the argument in Appendix~\ref{s.a.tools}.
For all $u,\gamma\geq 1$ and $t\in(0,1]$, 
\begin{align}
	\label{e.conti.pt-to-segment}
	\PP\Big[ \sup_{|x|\leq \gamma\sqrt{u}} \Big| \frac{1}{N^2T} \log \WW_N(0,0;t,x) \Big| > t^{1/6}u\Big] &\leq c {\gamma}^2 e^{-\frac{1}{c}N^3T u^{3/2}},
\\
	\label{e.conti.pt-to-line}
	\PP\Big[ \sup_{x\in\R} \Big\{ \frac{1}{N^2T} \log \WW_N(0,0;t,x) - |x| \Big\} > u\Big] &\leq c \,e^{-\frac{1}{c}N^3T u^{3/2}}.
\end{align}

We proceed to implement the steps outlined above.
Invoke the semigroup identity mentioned above and divide the result into $|x|\leq \gamma\sqrt{u}$ and $|x|>\gamma\sqrt{u}$, for some $\gamma$ to be specified later:
\begin{align}
	\label{e.conti.decomp}
	\ZZ_N(t_1,x_1) 
	= 
	\Big(\int_{|x|\leq\gamma\sqrt{u}} \d x \,NT\, + \int_{|x|>\gamma\sqrt{u}} \d x \,NT\, \Big) \,  \ZZ_N(t_0,x) \, \ZZ_N(t_0,x;t_1,x_1).
\end{align}

Let us bound the last integral in \eqref{e.conti.decomp}.
Write $\ZZ_N(t_0,x) = \ZZ_N(0,0;t_0,x) = \WW_N(0,0;t_0,x) \hk(Tt_0,NTx)$ and apply \eqref{e.conti.pt-to-line}.
As for $\ZZ_N(t_0,x;t_1,x_1)$, note that $\WW_N(t_0,\Cdot;t_1,x_1)$ has same law as $\WW_N(0,0;t_1-t_0,\Cdot)$ by the shift invariance and time-reversal symmetry; then apply \eqref{e.conti.pt-to-line}.
The result shows that, up to an affordable probability, the last integral in \eqref{e.conti.decomp} is bounded by
\begin{align}
	e^{N^2T\cdot 2u} \int_{|x|>\gamma\sqrt{u}} \d x\, NT  
	e^{N^2T(|x|+|x_1-x|)} \hk(Tt_0,NT x) \hk(T(t_1-t_0),NT(x_1-x)).
\end{align}
Recognize the last integral as
\begin{align}
	\label{e.integral=exp}
	\Ebm\big[e^{N\,(|\X(Tt_0)|+|NTx_1-\X(Tt_0)|)} \ind_{\{|\X(Tt_0)|>NT\gamma\sqrt{u}\}}\delta_0(\X(Tt_1)-NTx_1)\big],
\end{align}
where $\X=$ (standard BM).
Given that $|x_1|\leq R$, that $R$ is fixed, and that $u\geq1$, we fix a large enough $\gamma$ such that \eqref{e.integral=exp} is bounded by $\hk(Tt_1,NTx_1)\exp(-N^2Tu/((t_1-t_0)c_1))$, where we labeled the constant by $c_1$ for future references.
Recall that $\ZZ_N(t_1,x_1)=\WW_N(t_1,x_1)\hk(Tt_1,NTx_1)$ and that $\WW_N(t_1,x_1)$ has the same law as $\WW_N(t_1,0)$.
These properties and \eqref{e.conti.pt-to-segment} give, up to an affordable probability, $\hk(Tt_1,NTx_1)\leq\exp(N^2Tu)\ZZ_N(t_1,x_1)$.
We will choose $t_0$ in the next paragraph such that $t_1-t_0\leq 1/(4c_1)$.
Hence, up to an affordable probability,
\begin{align}
	\label{e.remainder}
	\text{(last integral in \eqref{e.conti.decomp})}
	\leq
	\hk(Tt_1,NTx_1) e^{N^2T u\,(2-\frac{1}{(t_1-t_0)c_1})}
	\leq
	e^{-\frac{N^2Tu}{(t_1-t_0)4c_1}}\ZZ_N(t_1,x_1).
\end{align}

Let us bound the first integral in \eqref{e.conti.decomp}.
Use \eqref{e.conti.pt-to-segment} for $t=t_1$ and for $t=t_2$, together with the time-reversal symmetry and shift invariance.
The result gives, up to an affordable probability,
$
	\ZZ_N(t_0,x;t_1,x_1) \leq \exp(N^2Tu(t_1-t_0)^{1/6}) \hk(T(t_1-t_0),NT(x_1-x))
$
and
$
	\exp(-N^2Tu(t_2-t_0)^{1/6}) \hk(T(t_2-t_0),NT(x_2-x)) \leq \ZZ_N(t_0,x;t_2,x_2).
$
Combining these bounds shows that the first integral in \eqref{e.conti.decomp} is bounded by
\begin{align}
	e^{2N^2Tu\,(t_2-t_0)^{1/6}} \int_{|x|\leq\gamma\sqrt{u}} \d x\, NT \, \ZZ_N(t_0,x) \frac{\hk(T(t_1-t_0),NT(x_1-x))}{\hk(T(t_2-t_0),NT(x_2-x))} \, \ZZ_N(t_0,x;t_2,x_2).
\end{align}
Fix $t_0$ by setting $(t_1-t_0) := \min\{1/R, 1/(4c_1), (t_2-t_1)^{6/13} + |x_2-x_1|^{6/7} \}$.
The term $1/R$ ensures $t_0\geq0$.
The term $1/(4c_1)$ fulfills the requirement $(t_1-t_0)\leq 1/(4c_1)$ from the previous paragraph.
The term $(t_2-t_1)^{6/13} + |x_2-x_1|^{6/7}$ ensures that the ratio of heat kernels in the last integral is bounded by $\exp(N^2Tcu\,(|t_2-t_1|^{1/13}+|x_2-x_1|^{1/7}))$, which is straightforward (though tedious) to verify.
Use the last bound in the integral, factor the bound out of the integral, release the integral to $x\in\R$, and use the semigroup identity to rewrite the integral as $\ZZ_N(t_2,x_2)$.
Doing so gives 
\begin{align}
	\label{e.firstintegralin...}
	\text{(first integral in \eqref{e.conti.decomp})} \leq e^{cN^2T(|t_2-t_1|^{1/13}+|x_2-x_1|^{1/7})u} \ZZ_N(t_2,x_2).
\end{align}

Given our choice of $(t_1-t_0)$, the right hand side of \eqref{e.remainder} is bounded by the last terms in \eqref{e.conti.goal}.
Hence, inserting \eqref{e.remainder}--\eqref{e.firstintegralin...} into \eqref{e.conti.decomp} gives the desired result \eqref{e.conti.goal}.

\subsection{Completing the proof of Theorem~\ref{t.main}}
\label{s.conti.tmain}
First, Proposition~\ref{p.conti} can be leveraged into a stronger result:
Given any fixed $(v,v')\in(0,1/13)\times(0,1/7)$, the event that
\begin{align}
	\label{e.conti.uniform}
	&\big| \hh_N(t_1,x_1)-\hh_N(t_2,x_2) \big | \leq (|t_2-t_1|^v+|x_2-x_1|^{v'})\, u,
\\
	\label{e.conti.uniform.}
	&\big| \tfrac{1}{N^2T}\log\big(\ZZ_N(0,x_0;t_1,x_1)\big/\ZZ_N(0,0;t_1,x_2)\big) \big| \leq (|x_0|^{v'}+|x_2-x_1|^{v'}) \, u,
\end{align}
for all $t_1,t_2\in[1/R,1]$ and $x_0,x_1,x_2\in[-R,R]$, holds up to an affordable probability.
This result is stronger than Proposition~\ref{p.conti} because the conditions \eqref{e.conti.uniform}--\eqref{e.conti.uniform.} are required to hold \emph{simultaneously} for all $t_1,t_2,x_0,x_1,x_2$ in their designated ranges.
One can use the standard argument in the proof of Kolmogorov’s continuity theorem to derive the stronger result from Proposition~\ref{p.conti}.
We refer to the proof of Proposition~3.4 in \cite{lin21} for an instance of such a derivation and omit it here.

Based on \eqref{e.conti.uniform}--\eqref{e.conti.uniform.}, let us show that $\hh_N(t,x)$ and $\hh_{N,\alpha}(t,x)$ approximate each other.
Write the latter as $(N^2T)^{-1}\log\int_{-\alpha}^{\alpha} \d x_0\,NT\, \ZZ_N(0,x_0;t,x)$ and apply \eqref{e.conti.uniform.} with $t_1=t$, $x_1=x_2=x$, and $v'=1/9$.
Up to an affordable probability, the integrand is sandwiched between $\exp(\pm N^2Tu\alpha^{1/9})\ZZ_N(0,0;t,x)=\exp(N^2T(\hh_N(t,x)\pm u\alpha^{1/9}))/\sqrt{T}$.
The last expression does not depend on $x_1$; factor it out of the integral.
After being simplified, the result gives
$
	|\hh_{N,\alpha}(t,x)-\hh_N(t,x)| \leq (N^2T)^{-1}\log(\alpha\sqrt{N^2T}) + u\alpha^{1/9}.
$
Set $r=u\alpha^{1/9}$.
We have, for every $t\in[1/R,1]$ and $r\geq \alpha^{1/9}$,
\begin{align}
	\label{e.soften.approx}
	\PP\Big[ \sup_{x\in[-R,R]} \big|\hh_{N,\alpha}(t,x)-\hh_{N}(t,x)\big| > r + \tfrac{1}{N^2T}\log(\alpha\sqrt{N^2T})  \Big]
	\leq
	c\, e^{-\frac{1}{c} N^3T r^{3/2} \alpha^{-1/6}}.
\end{align}

Lemma~\ref{l.conti.}, which was used in Section~\ref{s.shape}, follows immediately from the preceding results.
Part~\ref{l.conti..1} follows from \eqref{e.conti.uniform} and \eqref{e.soften.approx}.
Part~\ref{l.conti..2} follows from \eqref{e.conti.pt-to-segment}, the identity $\ZZ_{N}(t_1,x;t_2,y_2)=\WW_N(t_1,x;t_2,y_2)\hk(T(t_2-t_1),NT(y_2-x))$, the time-reversal symmetry, and the shift invariance.

Let us finish the rest of the proof of Theorem~\ref{t.main}.
By \eqref{e.soften.approx}, Corollary~\ref{c.ldp.soften} and Proposition~\ref{p.shape} hold with $\hh_{N,\alpha}$ being replaced by $\hh_N$.
They give \eqref{e.t.ldp} and the one-point version of \eqref{e.t.mps.}, respectively.
By the union bound, the one-point version automatically extends into a multipoint version:
For any $(t_1,x_1),\ldots,(t_k,x_k)\in[-1/R,1]\times[-R,R]$,
\begin{align}
	\label{e.t.mps.multi}
	\PP\Big[ 
		\max_{j=1,\ldots,k}\big|\hh_N(t_j,x_j)-\mps(t_j,x_j)\big| > \tfrac{1}{R}  \ \Big| \ \eventhh_{N,\delta}(\vechv) 
	\Big]
	\ll e^{-N^3T\cdot 0}.
\end{align}
Combining \eqref{e.conti.uniform} and \eqref{e.t.mps.multi} gives \eqref{e.t.mps.}.

\appendix

\section{Proof of technical tools}
\label{s.a.tools}
\begin{proof}[Proof of Lemma~\ref{l.ldp.upb}]
\ref{l.ldp.upb.1}\ 
It is straightforward to check from the definition of $\hfc{}$ that the map $\R^n\to\hvspacec(t,\vecxx): \vechv\mapsto (\hfc{,t,\vecxx,\vechv}(t,\xx_\cc))_{\cc=1}^n$ is continuous and agrees with the identity map on $\hvspacec(t,\vecxx)$.
Also, recall from Property~\ref{property.legendre} that $\ratekpz(t,\vecxx,\Cdot):\hvspacec(t,\vecxx)\to[0,\infty)$ is convex and continuous.
Next, let us verify that $(\ratekpz\circ\hfc{})(t,\vecxx,\Cdot):\R^n\to[0,\infty)$ is convex.
Take any $\vechv,\vechvi\in\R^n$ and $u\in[0,1]$, and set $\hfc{,1}:=u\hfc{,t,\vecxx,\vechv}+(1-u)\hfc{,t,\vechvi}$ and $\hfc{,2}:=\hfc{,t,\vecxx,u\vechv+(1-u)\vechvi}$.
Note that $\hfc{,1}(\xx_\cc) \geq \max\{ u\hv_\cc+(1-u)\hvi_\cc, \parab(t,\xx_\cc) \}$, for all $\cc$.
Since the hypograph of $\hfc{,2}$ is the convex hull of $\{(x,\parab(t,x))\}_{x\in\R}\cup\{(\xx_\cc,\hv_\cc)\}_{\cc=1}^n$, we have $\hfc{,1} \geq \hfc{,2}$.
From this property and the fact that $\hfc{,1}$ and $\hfc{,2}$ are concave, it is not hard to verify that $\ratekpz((\hfc{,1}(\xx_\cc))_{\cc=1}^n)\geq\ratekpz((\hfc{,2}(\xx_\cc))_{\cc=1}^n)$, which gives the convexity of $(\ratekpz\circ\hfc{})(t,\vecxx,\Cdot)$.
Collecting the preceding properties gives that $(\ratekpz\circ\hfc{})(t,\vecxx,\Cdot):\R^n\to[0,\infty)$ is convex and continuous on $\R^n$, and is strictly convex and agrees with $\ratekpz$ on $\hvspacec$.
The desired result hence follows.

\ref{l.ldp.upb.2}\
First, note that $\hfc{}$ and $(\ratekpz\circ\hfc{})$ can be realized as $\hf{}$ and $\ratekpz$ through a reduction procedure.
For $\vechv\in\R^n$, consider $\mathsf{Redu}(\vechv):=\{\cc\in\{1,\ldots,n\} : \hfc{,t,\vecxx,\vechv}(\xx_\cc) = \hv_{\cc} \}$, which reduces the full index set $\{1,\ldots,n\}$ to a subset that suffices for characterizing $\hfc{}$.
For $\vecxx\in\R^n$ and $C\subset\{1,\ldots,n\}$, write $\vecxx_C := (\vecxx_a)_{a\in C}\in\R^{C}$ and similarly for $\vechv_C $.
It is not hard to check that, for all $\vechv\in\R^n$,
\begin{align}
	\label{e.ratekpz.reduced}
	\vechv_{\mathsf{Redu}(\vechv)}\in\hvspacec(t,\vecxx_{\mathsf{Redu}(\vechv)}),
	\qquad
	\big(\ratekpz\circ\hfc{}\big)(t,\vecxx,\vechv)
	=
	\ratekpz\big(t,\vecxx_{\mathsf{Redu}(\vechv)},\vechv_{\mathsf{Redu}(\vechv)}\big).
\end{align}

We now prove \eqref{e.l.ldp.upb.2}.
The first step is to truncate the set $F$.
By Properties \ref{property.limit}--\ref{property.legendre} in Section~\ref{s.momt.tsai2023high}, for every $\cc$ and $R>1+\xx_{\cc}^2/(2t)$, $\limsup_{\alpha\to 0}\limsup_{N\to\infty} (N^3T)^{-1}\log\PP\big[ \hh_{N,\alpha}(t,\vecxx_\cc) \geq R \big]< -\ratekpz(t,\xx_\cc,R-1,\ldots,R-1)$.
The last quantity tends to $-\infty$ as $R\to\infty$.
Consider the truncated set $ F':=\cap (-\infty,R]^n$.
Once we have proven \eqref{e.l.ldp.upb.2} with $F'$ replacing $F$, sending $R\to\infty$ recovers \eqref{e.l.ldp.upb.2} itself.
Next, for each nonempty $C\subset\{1,\ldots,n\}$, consider $F'_C := \{ \vechv_{C} : \vechv\in F', \mathsf{Redu}(\vechv)=C \}$, and define 
$
	F'_\emptyset := \{ \vechv : \vechv\in F', \mathsf{Redu}(\vechv)=\emptyset \}.
$
We have $F' \subset (\cup_{\emptyset\neq C\subset\{1,\ldots,n\}} (F'_C\times \R^{\{1,\ldots,n\}\setminus C}))\cup F'_\emptyset$.
If $F'_\emptyset$ is nonempty, $F'$ contains an $\vechv$ such that $\hv_\cc<-\xx_\cc^2/(2t)$ for all $\cc$.
In this case, $\inf_{F'}(\ratekpz\circ\hfc{})=0$, so the desired result follows trivially.
We assume $F'_\emptyset=\emptyset$ hereafter.
Under this assumption, the last inclusion relation gives
\begin{align}
	\label{e.l.ldp.upb.2.}
	\PP\big[ (\hh_{N,\alpha}(t,\vecxx_\cc))_{\cc=1}^n \in F' \big]
	\leq
	\sum_{C} \PP\big[ (\hh_{N,\alpha}(t,\vecxx_\cc))_{\cc\in C} \in F'_C \big],
\end{align} 
where the sum runs over all nonempty subset $C$ of $\{1,\ldots,n\}$.
Note that $F'_C$ is closed (because $F'$ is), is contained in $\hvspacec(t,\vecxx_C)$ (because of the first relation in \eqref{e.ratekpz.reduced}), and is bounded (thanks to the truncation).
Hence $F'_C$ is a compact subset of $\hvspacec(t,\vecxx_C)$.
Given this property, Corollary~\ref{c.ldp.soften} together with the continuity of $\ratekpz$ gives
\begin{align}
	\label{e.l.ldp.upb.2..}
	\limsup_{\alpha\to 0} 
	\limsup_{N\to\infty}
	\frac{1}{N^3T} \PP\big[ (\hh_{N,\alpha}(t,\vecxx_\cc))_{\cc\in C} \in F'_C \big] \leq - \inf_{F'_C} \, \ratekpz.
\end{align}
From the second relation in \eqref{e.ratekpz.reduced} and the definition of $F'_C$, it is not hard to verify the identity $\min_{C}\inf_{F'_C} \ratekpz = \inf_{F'}(\ratekpz\circ\hfc{})$.
Combining this identity with \eqref{e.l.ldp.upb.2.}--\eqref{e.l.ldp.upb.2..} gives \eqref{e.l.ldp.upb.2} with $F'$ in place of $F$.
This completes the proof.
\end{proof}

\begin{proof}[Proof of \eqref{e.fkg}]
Let us prove the analog of \eqref{e.fkg} where $\ZZ_{N,\alpha}$ is replaced by $\ZZ_{N}$; the proof of \eqref{e.fkg} is the same but requires heavier notation.
The scaling by $N,T$ is irrelevant, so consider $\ZZ(t,x):=\ZZ_N(t/N,x/(NT))$ and $\ZZ(t_0,x_0;t,x)$ similarly.
By the FKG inequality, for any $f\in\Csp(\R)$,
\begin{subequations}
\label{e.fkg.}
\begin{align}
	\PP&\big[ \ind_{\{ \inf_{x\in[y,y']}(\ZZ(t_1,y_1;t_2,x) -f(x))\geq 0 \}} \cdot \prod\nolimits_{j}\ind_{\{ \ZZ(t,x_{j}) \geq r_j \}} \big]
\\
	&\geq
	\PP\big[ \inf\nolimits_{x\in[y,y']}(\ZZ(t_1,y_1;t_2,x) -f(x))\geq 0  \big] \cdot \PP\big[\prod\nolimits_{j}\ind_{\{ \ZZ(t,x_{j}) \geq r_j \}} \big].
\end{align}
\end{subequations}
To see why this follows from FKG, pretend for a moment that the spacetime white noise $\noise$ is function-valued and observe, from the Feynman--Kac formula \eqref{e.ZZ.ptp}, that $\ZZ$ appears to be a monotone function of $\noise$.
To make this observation rigorous, recall from \cite{alberts2014intermediate} that $\ZZ$ can be obtained as the continuum limit of the partition function of a discrete polymer.
In the discrete polymer, the spacetime white noise $\noise$ is replaced by an iid field over $\Z_{>0}\times\Z$, and the partition function is truly an increasing function of the iid field.
By the FKG inequality, the analog of \eqref{e.fkg.} for the discrete polymer holds, and taking the continuum limit of the analog gives \eqref{e.fkg.}.
Now, assume that $x_1,\ldots,x_m$ are distinct and integrate \eqref{e.fkg.} over $r_1,\ldots,r_m\in(0,\infty)$.
The result gives
$
	\EE[ \ind_{\{ \inf_{x\in[y,y']}(\ZZ(t_1,y_1;t_2,x) -f(x))\geq 0 \}}\cdot \prod_{j}\ZZ(t,x_{j}) ]
	\geq
	\PP[ \inf_{x\in[y,y']}(\ZZ(t_1,y_1;t_2,x) -f(x))\geq 0 ] \cdot \EE[\prod_{j}\ZZ(t,x_{j}) ].
$
By the continuity of $\ZZ(t,x)$ in $x$, the same inequality holds even when $x_1,\ldots,x_m$ are not distinct.
Taking the complement of the inequality gives the desired result.
\end{proof}

\begin{proof}[Proof of \eqref{e.conti.pt-to-segment}--\eqref{e.conti.pt-to-line}]
Write $\WW_N(t,x):=\WW_N(0,0;t,x)$ to simplify notation.
We begin by establishing the point-to-point tail bounds:
For all $(t,x)\in(0,1]\times\R$,
\begin{align}	
	\label{e.conti.pt-to-pt.up}
	&&&&\PP\big[ \tfrac{1}{N^2T} \log \WW_N(t,x) > u \big] 
	&\leq 
	c \, e^{-\frac{1}{c} N^3T u^{3/2}t^{-1/2}}
	\leq
	c \, e^{-\frac{1}{c} N^3T u^{3/2}t^{-1/3}},
	&&
	u \geq t,
\\
	\label{e.conti.pt-to-pt.lw}
	&&&&\PP\big[ \tfrac{1}{N^2T} \log \WW_N(t,x) < -u \big] 
	&\leq 
	c \, e^{-\frac{1}{c} N^3T u^{2}t^{-1/2}}
	\leq
	c \, e^{-\frac{1}{c} N^3T u^{3/2}t^{-1/3}},
	&&
	u \geq t^{1/3}.
\end{align}
First, by the shift invariance of $\WW_N$, we may assume $x=0$.
The bound \eqref{e.conti.pt-to-pt.lw} follows by combining \cite[Thm.\ 1.7]{das21iter} and \cite[Prop.\ 2.11]{corwin2021kpz}.
For $Tt\leq$ (fixed constant), the bound \eqref{e.conti.pt-to-pt.up} is proven in \cite[Thm.\ 1.4, Prop.\ 3.1]{das21iter} by bounding the positive-integer moments.
Here, we need \eqref{e.conti.pt-to-pt.up} to hold for all $t\in(0,1]$, which we achieve by deriving a moment bound that holds for all time.
Let $\WW(t,x):=\WW_N(t/T,x/(NT))$ and use the Feynman--Kac formula~\eqref{e.feynmankac} to get
$
	\EE[\WW(t,0)^k]
	=
	\Ebb[ \exp(\int_0^t \d s \sum_{\ii<\jj}\delta_0(\X_{\ii}-\X_{\jj})) ],
$
where, under $\Ebb$, $\X_1,\ldots,\X_k$ are independent Brownian Bridges (BBs) connecting $(0,0)$ and $(t,0)$.
The BBs are semimartingales, with $\d \X_\ii = -(\X_{\ii}/(t-s))\, \d s + \d \bm_\ii$, where $\bm_1,\ldots,\bm_k$ are independent BMs, so by the It\^{o}--Tanaka--Meyer formula,
\begin{align}
	\int_0^t\d s \,\sum_{\ii<\jj}\delta_0(\X_{\ii}-\X_{\jj})
	=&
	\int_0^t\d s \sum_{\ii=1}^{k} \,\frac{\X_{\ii}}{t-s} \sum_{\jj:\jj\neq \ii} \frac12 \sgn(\X_{\ii}-\X_{\jj})
\\
	&-
	\int_0^t \sum_{\ii=1}^{k} \d\bm_\ii \sum_{\jj:\jj\neq \ii} \frac12 \sgn(\X_{\ii}-\X_{\jj}),
\end{align}
where $\sgn(x):=(x/|x|)\ind_{x\neq 0}$.
Insert the last identity into the last expectation and use the Cauchy--Schwarz inequality to separate the contribution of the two integrals.
For the integral wrt to $\d s$, bound the sum over $\jj$ by $2\cdot k/2$, use the independence of the BBs, and perform a time reversal to turn $1/(t-s)$ into $1/s$.
For the integral wrt to $\sum_{\ii}\d\bm_\ii$, note that the integral is a martingale of quadratic variation $k(k^2-1)t/3$.
Doing so gives
$
	\EE[\WW(t,0)^k]
	\leq
	\Ebb[ \exp( k\int_0^t \d s |\X|/s) ]^{k/2} \exp(k^3t/12).
$
Replace $\Ebb$ with $\Ebm$ in the last expectation; doing so only makes the result larger.
It is not hard to show that
$ 
	\Ebm[ \exp( k\int_0^t \d s |\X|/s) ]
	\leq
	\Ebm[\exp(2k|\X(t)|)],
$
for example by Taylor expanding the left hand side and bounding the resulting terms with the aid of Minkowski's inequality.
Bound the last expectation by $2\Ebm[\exp(2k\X(t))]=2\exp(2k^2t)$.
Altogether, $\EE[\WW(t,0)^k]\leq 2\exp((13/12)k^3t)$, for all $t>0$ and $k\in\Z_{>0}$.
Using this and the exponential Chebyshev inequality gives \eqref{e.conti.pt-to-pt.up}.

The next step is to leverage the point-to-point bounds \eqref{e.conti.pt-to-pt.up}--\eqref{e.conti.pt-to-pt.lw} into the point-to-line bounds \eqref{e.conti.pt-to-segment}--\eqref{e.conti.pt-to-line} by the arguments in \cite{corwin2021kpz}.
We begin with \eqref{e.conti.pt-to-segment}.
By the shift invariance of $\WW_N$, we replace the interval in \eqref{e.conti.pt-to-segment} with $[0,2\gamma\sqrt{u}]$ and call it $[0,K]$.
Divide the interval into equally spaced subintervals of length $\zeta$, to be specified later, and write $\zeta_j:=j\zeta$.
Given \eqref{e.conti.pt-to-pt.lw}, we apply the coupling argument in \cite{corwin2021kpz} that leads to Eq.\ (20) there (in the proof of Prop.\ 4.2, see Figure~1 there).
When applied to the intervals $[\zeta_{j-1},\zeta_j]$ and $[\zeta_j,\zeta_{j+1}]$, the argument gives, for all $u\geq t^{1/3}$,
\begin{align}
	\label{e.cgh.1.}
	\PP\Big[ \sup_{[\zeta_{j},\zeta_{j+1}]}\!\! \frac{1}{N^2T}\log\WW_N(t,x) \geq u \Big]
	\leq
	2\PP\Big[ \frac{1}{N^2T}\log\WW_N(t,\zeta_j) \geq A_1 \Big]
	+
	c\, e^{-\frac{1}{c} N^3T u^{3/2}t^{-1/3}},
\end{align}
where $A_1:=u/4 + (\zeta_{j}^2-\zeta_{j+1}^2)/(2t)$.
Apply \eqref{e.conti.pt-to-pt.up} to bound the second last term in \eqref{e.cgh.1} to get
\begin{align}
	\label{e.cgh.1}
	\PP\Big[ \sup_{[\zeta_{j},\zeta_{j+1}]} \frac{1}{N^2T}\log\WW_N(t,x) \geq u \Big]
	\leq
	c \, e^{-\frac{1}{c} N^3T (A_1)_+t^{-1/3}}
	+
	c\, e^{-\frac{1}{c} N^3T u^{3/2}t^{-1/3}}.
\end{align}
Next, for the lower-tail bound, we apply the stochastic monotonicity in \cite[Lem.\ 2.4]{corwin2021kpz}, which is based on \cite[Lem.\ 2.6--2.7]{corwin2016kpz}.
When applied to the interval $[\zeta_j,\zeta_{j+1}]$, the monotonicity together with \eqref{e.conti.pt-to-pt.lw} gives, for all $u\geq t^{1/3}$,
\begin{align}
	\label{e.cgh.2}
	\PP\Big[ \inf_{[\zeta_{j},\zeta_{j+1}]} \frac{1}{N^2T}\log\WW_N(t,x) \leq -u \Big]
	\leq
	c\, e^{-\frac{1}{c} N^3T (A_2)_+^2/\zeta}
	+
	c\, e^{-\frac{1}{c} N^3T u^{3/2}t^{-1/3}},
\end{align}
where $A_2:= u/2-(\zeta_{j+1}^2-\zeta_{j}^2)/(2t) $.
Now, choose $\zeta=(100 K)^{-1}t^{4/3}$.
This way, the terms $A_1$ and $A_2$ are at least $u/c$.
Take the union bound of \eqref{e.cgh.1}--\eqref{e.cgh.2} over $j=0,\ldots,K/|\zeta|$.
Doing so gives, for all $u'\geq t^{1/3}$ and $K \geq 1$,
\begin{align}
	\label{e.conti.pt-to-segment.}
	\PP\Big[ \sup_{|x|\leq K} \Big| \frac{1}{N^2T} \log \WW_N(0,0;t,x) \Big| > {u'}\Big] 
	&
	\leq c K^2 t^{-4/3} e^{-\frac{1}{c} N^3T {u'}^{3/2}t^{-1/3}}.
\end{align} 
This gives \eqref{e.conti.pt-to-segment} by substituting $u'=t^{1/6}u$ and $K=\gamma\sqrt{u}$.
The proof of \eqref{e.conti.pt-to-line} proceeds similarly, by dividing the whole $\R$ into subintervals of length $\zeta=(100)^{-1}t^{4/3}$.
\end{proof}

\section{A sign of symmetry breaking}
\label{s.a.symmetrybreaking}
Here we consider the two-delta-like initial condition $\ZZ_{N,\alpha}(0,\Cdot) = \ind_{[-\alpha-1,-1+\alpha]}+\ind_{[-\alpha+1,1+\alpha]}$ and analyze the moments $\EE[\ZZ_{N,\alpha}(1,0)^{N\mm}]$ under \eqref{e.scaling}.
By the linearity of the SHE, 
\begin{align}
	\label{e.symmetrybreaking.expand}
	\EE[\ZZ_{N,\alpha}(1,0)^{N\mm}]
	=
	\hspace{-15pt}
	\sum_{\mm_-\in[0,\mm]\cap\frac{1}{N}\Z} \binom{N\mm}{N\mm_-}
	\EE\big[\ZZ_{N,\alpha}(0,-1;1,0)^{N\mm_-}\, \ZZ_{N,\alpha}(0,+1;1,0)^{N(\mm-\mm_-)}\big].
\end{align}
Recall $\WW_N$ and its time-reversal symmetry from Section~\ref{s.conti.}.
Use the symmetry to turn the last expectation into $\EE[\ZZ_{N,\alpha}(0,- 1)^{N\mm_-}\ZZ_{N,\alpha}(0,1)^{N (\mm-\mm_-)}]$, apply Property~\ref{property.limit} in Section~\ref{s.momt.tsai2023high}, and note that the binomial factor in \eqref{e.symmetrybreaking.expand} is negligible compared to $\exp(O(N^3T))$ under \eqref{e.scaling}.
We have
\begin{align}
	\label{e.symmetrybreaking.limit}
	\EE[\ZZ_{N,\alpha}(1,0)^{N\mm}]
	\sim
	\sup
	\momshe\big( 0 \xrightarrow{1} ((-1,1),(\mm_-,\mm_+)) \big),
\end{align}
where the supremum runs over all $\mm_-,\mm_+\in[0,\mm]$ with $\mm_-+\mm_-=\mm$.
As will be shown in the next paragraph, the supremum is achieved at $\mm_-=\mm$ and $\mm_+=\mm$ and at these two places only.
Looking back at \eqref{e.symmetrybreaking.expand}, we see that the moment is dominated by the contribution from either one of the deltas, which hints at the symmetry breaking stated in \eqref{e.symmetry.breaking}.

Let us analyze the supremum in \eqref{e.symmetrybreaking.limit}.
Recall Property~\ref{property.legendre}, let $(\hv_-,\hv_+)$ be the Legendre-dual variable of $(\mm_-,\mm_+)$, and let $\optimal_1, \optimal_2$ be the corresponding noise corridors, namely the corridors for $\{(\xx_\cc,\hv_\cc)\}_\cc=\{(-1,\hv_-),(+1,\hv_+)\}$.
In the extreme case $(\mm_-,\mm_+)=(\mm,0)$, there is only one corridor, which is the line that connects $(t,x)=(1,-1)$ and $(t,x)=(0,0)$.
We let $\optimal_-$ denote this linear corridor, and do similarly for $\optimal_+$.
Recall $\mom$ from \eqref{e.mom}.
By \cite[Thm.\ 2.3]{tsai2023high}, $\momshe( 0 \xrightarrow{\scriptscriptstyle 1} ((-1,1),(\mm_-,\mm_+)) )=\mom_{[0,1]}(\mm_-\delta_{\optimal_1}+\mm_+\delta_{\optimal_2})$.
More explicitly,
\begin{subequations}
\label{e.symmetrybreaking}
\begin{align}
	\label{e.symmetrybreaking.1}
	\momshe\big( 0 \xrightarrow{1} ((-1,1),(\mm_-,\mm_+)) \big)
	=
	&\frac{1}{24} \Big( (1-s_\text{merge}) \, \big( (\mm_+)^{3} + (\mm_-)^{3} \big) + s_\text{merge} \mm^3 \Big)
\\
	\label{e.symmetrybreaking.2}
	&- \frac12 \mm_+ \int_0^1 \d s \, \big(\tfrac{\d~}{\d s}\optimal_{1}\big)^2 - \frac12 \mm_- \int_0^1 \d s \, \big(\tfrac{\d~}{\d s}\optimal_{2}\big)^2,
\end{align}
\end{subequations}
where $s_\text{merge}$ is the first time when $\optimal_1$ and $\optimal_2$ merge, viewed in backward time.
Indeed, the right hand side of \eqref{e.symmetrybreaking.1} is $\leq\mm^3/24$.
Given that $\optimal_1(0)=-1$, $\optimal_1(1)=0$, $\optimal_2(0)=1$, and $\optimal_2(1)=0$, we have $\eqref{e.symmetrybreaking.2} \leq- \frac12 \mm \int_0^1 \d s \, (\frac{\d~}{\d s}\optimal_{-})^2 = -\frac12 \mm \int_0^1 \d s \, (\frac{\d~}{\d s}\optimal_{+})^2$.
These inequalities are strict except in the extreme cases $(\mm_-,\mm_+)=(\mm,0)$ and $(\mm_-,\mm_+)=(0,\mm)$.
Therefore, the supremum in \eqref{e.symmetrybreaking.limit} is achieved at $\mm_-=\mm$ and $\mm_+=\mm$ and at these two places only.

\bibliographystyle{alpha}
\renewcommand{\bibliofont}{\tiny}
\bibliography{kpz-ldp,line-ensemble}
\end{document}